\documentclass[12pt,a4paper]{article}
\usepackage[margin=25mm]{geometry}
\usepackage[utf8]{inputenc}

\usepackage{amsmath,amsfonts,amssymb,bef_alex,%
            datetime,todonotes,bm,relsize,tikz}

\numberwithin{equation}{section}
\numberwithin{figure}{section}

\usepackage[notref,notcite,color, final]{showkeys}

\makeatletter
\renewcommand*\env@cases[1][1.2]{%
  \let\@ifnextchar\new@ifnextchar
  \left\lbrace
  \def\arraystretch{#1}%
  \array{@{}c@{\quad}l@{}}%
}
\makeatother

\renewcommand{\tilde}{\widetilde}

\newcommand{\LB}{\lambda_\rmB}

\newcommand{\Rd}{{\R^d}}
\newcommand{\EB}{\calE_{\rm{B}}}
\newcommand{\Imix}{\calI_\Lambda}
\newcommand{\ImOne}{\calI_{\Lambda,1}}
\newcommand{\ImSec}{\calI_{\Lambda,2}}
\newcommand{\Ipmix}{\calI_{p,\Lambda}}
\newcommand{\HELL}{{\mathsf{He}}}

\newcommand{\IFisher}{\calI_{\rm{Fisher}}}
\newcommand{\IpFisher}{\calI_{p,\rm{Fisher}}}
\newcommand{\Dreact}{\calD_{\rm{react}}}
\newcommand{\Dpreact}{\calD_{p,\rm{react}}}

\newcommand{\AAA}{} 
\newcommand{\EEE}{\color{black}}

 \newcommand{\TODO}[1]{}

\begin{document}

\title{Convergence to self-similar profiles in reaction-diffusion systems%
    \thanks{Research partially supported by DFG via SFB\,910 ``Control of self-organizing
nonlinear systems'' (project no.\,163436311), subproject A5 ``Pattern
formation in coupled parabolic systems''.} }

\author{Alexander Mielke\thanks{Weierstraß-Institut f\"ur Angewandte
   Analysis und Stochastik, Mohrenstr.\,39, 10117 Berlin and  Humboldt
 Universit\"at zu Berlin, Institut f\"ur Mathematik, Berlin,
 Germany.}{}\ \ and Stefanie Schindler\thanks{Weierstraß-Institut f\"ur Angewandte
   Analysis und Stochastik, Mohrenstr.\,39, 10117 Berlin}
}

\date{5 April 2023}
 
\maketitle
%\footnotetext{\quad \hfill   \tiny \usdate\today, \currenttime\,h}

\begin{abstract} We study a reaction-diffusion system on the real line, where
  the reactions of the species are given by one reversible reaction pair
  $\alpha X_1 \rightleftharpoons \beta X_2$ satisfying the mass-action
  law.  We describe different positive limits at $x\to -\infty$ and
  $x\to +\infty$ and investigate the long-time behavior. Rescaling space and
  time according to the parabolic scaling with $\tau= \log(1{+}t)$ and
  $y= x/\sqrt{1{+}t}$, we show that solutions converge exponentially for
  $\tau \to \infty$ to a similarity profile. In the original
  variables, these profiles correspond to asymptotically self-similar behavior
  describing the phenomenon of diffusive mixing of
  the different states at infinity.
 
  Our method provides global exponential convergence for all initial states
  with finite relative entropy. For the case $\alpha = \beta \geq 1$ we can
  allow for self-similar profiles with arbitrary equilibrated states,
  while for $\alpha > \beta \geq 1$ we need to assume that the two states at
  infinity are sufficiently close such that the self-similar profile is
  relative flat. \medskip

\noindent
\emph{Keywords:} Mass-action kinetics, relative Boltzmann entropy,
infinite-mass systems, energy-dissipation estimates, self-similar profiles.   \smallskip

\noindent
\emph{MSC:} 
35K57 % Reaction-diffusion equations 
35C06 % Self-similar solutions to PDEs
35B45 % A priori estimates in context of PDEs

\end{abstract}

\section{Introduction}
\label{se:intro}

For a nonlinear coupled reaction-diffusion systems with mass-action kinetics
satisfying a detailed balance condition, the long-time behavior of its
solutions is investigated. While there already exists a wide variety of
literature for these systems posed on bounded domains \cite{Alik79AIPR,
  Smol94SWRD, DesFel06EDTE, DesFel07EMRD, BLMV14LFBD, MiHaMa15UDER}, much less is known if
the underlying domain is chosen to be the whole space $\Omega = \R^d$, see
e.g.\ \cite{HHMM18DEER} for the case with finite mass. We will see that in the
case of unbounded domains and infinite mass, 
similarity profiles can be a crucial tool to describe their asymptotic behavior
qualitatively. Here we follow the ideas on \emph{diffusive mixing} as
developed in \cite{BriKup92RGGL, ColEck92SPSG, GalMie98DMSS}; however, our
approach is completely different. Instead of 
doing a local analysis of the relevant similarity profile, whose
existence is established in \cite{MieSch21?ESPS}, we use the relative Boltzmann
entropy and derive exponential convergence globally, i.e.\ for all initial
conditions with finite relative entropy. Thus, our work is closer to
\cite{BLMV14LFBD} which derives global energy-dissipation estimates for
suitable relative entropies to study exponential convergence to 
(non-equilibrium) steady states on bounded
domains but with nontrivial prescribed Dirichlet boundary  data. To the best of
the authors' knowledge, the present work is the first relative-entropy approach
to systems with infinite mass. 

Asymptotic self-similar behavior for scalar, nonlinear diffusion equations is a
classical theory in the case of finite mass, see e.g.\ \cite{CarTos00ALDS,
  Vazq07PMEM} and the references therein. The case of inifinite mass was
initiated in \cite{Pele71ABSP} and extended in \cite{VanPel77ABSN, Bert82ABSN},
where the theory was based on comparison principles. Systems of partial
differential equations with different limits at $x\to -\infty$ and $x \to
\infty$ are studied in \cite{MarPan01DPSC} in the context of adiabatic gas flow
through a porous medium. The Convergence to asymptotic profiles is established
using local estimates with weighted Sobolev norms. 

For our model, we consider two species $X_1$ and $X_2$ on $\Omega = \R^d$
and denote their concentrations at time $t >0$ and at position $x \in \R^d$ by
$\tilde \bfu(t,x) = \big(\tilde u(t,x),\tilde v(t,x)\big)^\top$. The species
diffuse with diffusion coefficients $d_1,d_2 >0$ and interact through the
single reversible chemical reaction pair
$\alpha X_1 \overset{k}{\rightleftharpoons} \beta X_2$ with each other.  Here
$\alpha,\beta \geq 1$ are the stoichiometric coefficients and $k >0$ denotes
the reaction \AAA strength. \EEE The change of the concentrations can be
described by the corresponding reaction-diffusion system
\begin{equation} \label{eq:RDS.vector}
\tilde \bfu_t = \bfD \, \Delta \tilde \bfu+ \bfR(\tilde \bfu) 
	\quad \text{ for } t > 0,  \; x \in  \R^d,
\end{equation}
with diffusion matrix $\bfD = \mathrm{diag}(d_1,d_2)$ and, by using the law of
mass action, reaction term
\[
  \bfR(\tilde \bfu) = k (\tilde{v}^\beta {-} \tilde{u}^\alpha) \binom{\alpha}{-\beta}.
\]
Additionally, we require that the solutions are prescribed at infinity by
states that are in reactive equilibrium. In the case $d=1$ this simply means
that we require that the solutions are in equilibria at both sides of infinity,
thus we have $\tilde \bfu(t,\pm \infty) = (A_\pm^\beta , A_\pm^\alpha)^\top$
for two given constants $A_-,A_+ \geq 0$. Here we use that our system
\eqref{eq:RDS.vector} has the special property that it possesses a continuum of
constant solutions.  In contrast to reaction-diffusion systems with a finite
number of constant steady states, where the typical long-time behavior is given
by traveling waves or pulses (see for example \cite{Smol94SWRD, VoVoVo94TWSP,
  Volp14EPDE}), we show that the solutions converge to so-called self-similar
profiles when time goes to infinity. This \emph{asymptotically self-similar
  behavior} is called \textit{diffusive mixing} in \cite{ColEck92SPSG,
  GalMie98DMSS, MiScUe01SDDE}, where it was studied for the special system of
the real Ginzburg-Landau equation.

The existence of relevant self-similar profiles is established in
\cite[Sec.\,5]{MieSch21?ESPS}.  In the present paper, we focus on
proving the convergence towards these profiles. Our analysis is based on two
crucial steps:
\begin{itemize}
\itemsep=-0.3em
\item%[(i)] 
instead of the physical variables $(t,x)$ 
we use the parabolic scaling variables $(\tau,y)$ defined via 
$ \tau=\log(t{+}1)$ and $y=x/\sqrt{t{+}1}$ and

\item%[(ii)] 
we derive energy-dissipation estimates for a relative entropy (of
Boltzmann type).   
\end{itemize}

Doing the transformation $\bfu(\tau,y) =\tilde \bfu(t,x)$,  the
scaled system reads 
\begin{equation} \label{eq:RDS.vector.scaled}
\bfu_\tau = \bfD \, \bfu_{yy}+ \frac{y}{2} \bfu_y + \ee^\tau \, \bfR(\bfu) \quad
\text{ with } \bfu(\tau,\pm \infty) = (A_\pm^\beta , A_\pm^\alpha)^\top.
\end{equation}
Note that we cannot scale the size of the variables $u$ and $v$ because of
the fixed boundary conditions. Clearly, the parabolic scaling is good for the 
diffusion term, but we see that an additional time dependent factor
appears in front of the 
reaction term. As the prefactor is exponentially growing in time, it forces the
reaction to decay fast in order to equilibrate the whole system. This is
indeed established in \cite{GalSli22DREE} for the case $ \alpha=2 $ and
$ \beta =1 $. 

The equation for the asymptotic profile $\bfU=(U,V):\R\to \R^2$ can
be motivated as follows. We 
 consider \eqref{eq:RDS.vector.scaled} with the constraint
$\bfR(\bfu) = 0$ and replace the limit $\ee^\tau \, \bfR(\bfu)
\to$ ``$ \infty \bold 0$'' by a vector-valued Lagrange multiplier $\bflambda =
\Lambda(y)\binom{\alpha }{-\beta }$. This leads to the profile equation for $\bfU$ in
the form   
\begin{equation}
  \label{eq:I.ProfEqn}
  0 = \bfD \bfU''(y) + \frac y2 \bfU'(y) + \Lambda(y) \binom{\alpha }{-\beta}, \quad
  0=\bfR(\bfU(y)), \quad \bfU(\pm\infty) = \binom{A_\pm^\beta }{ A_\pm^\alpha } .  
\end{equation}
We refer to  \cite{MieSch21?ESPS} for the treatment of this and more
general profile equations, where the approach from \cite{GalMie98DMSS} relying
on the theory of monotone operators is generalized to the vector-valued
case. 

The idea is to use entropy estimates to study the asymptotic 
behavior of solutions of reaction-diffusion systems with
mass-action kinetic (and detailed balance). It traces back to the works of 
\cite{Grog83ABSC, Grog92FEEA, GlGrHu96FEDR} and has been refined by many 
authors in recent years, see for example
\cite{DesFel06EDTE,DesFel07EMRD,MiHaMa15UDER, FelTan17EECE, PiSuZo17ABRD}.
The usual strategy is to take a relative entropy given by
\[
\calE_\phi (\bfu(\tau) \,\vert \, \bfU) = \int_\Omega \sum_{j=1}^{j_*} \phi \big(
u_j(\tau,y) / U_j(y)\big) U_j(y) \dd y
\]
for a convex entropy function $\phi$ satisfying $\phi(\rho)> \phi(1) = 0$ for all
$\rho\neq 1$, and to show that $\calE_\phi$ is a Lyapunov function, i.e.\ 
along solutions a so-called entropy-dissipation relation
\[
\frac{\rmd}{\rmd \tau} \calE_\phi(\bfu(\tau) \,\vert \, \bfU) = -
\calD_\phi(\bfu(\tau)) \leq 0
\]
holds with a non-negative dissipation functional $\calD_\phi$.  If one can establish a
lower bound $\calD_\phi \geq \mu \calE_\phi$ with $\mu>0$, then Gr\"onwall's lemma
gives $\calE_\phi(\bfu(\tau)|\bfU) \leq \ee^{-\mu \tau} \calE_\phi(\bfu(0)|\bfU)$ which
implies convergence of $\bfu(\tau)$ to $\bfU$ because $\calE_\phi(\bfu|\bfU)=0$
if and only if $\bfu=\bfU$.

Due to the mass-action kinetics, the relative Boltzmann entropy $\calE_{\rmB}$
with Boltzmann function $\phi=\LB: z\mapsto z \log z-z+1$ is the only choice in
order to obtain the right sign for the dissipation term coming from the
reaction with $\alpha\neq \beta$. This can be seen in more detail in
Proposition \ref{prop:Dissipation} and will be important in Section
\ref{se:GeneralCase}. We will see in Section \ref{su:al=be.GenEntropy} that for
$\alpha=\beta$ other entropies are useful.

The classical studies on the long-time behavior of solutions for the unscaled
system \eqref{eq:RDS.vector} in a bounded domain $\Omega$ (see.\ e.g.\
\cite{Grog83ABSC, DesFel06EDTE, Miel17UEDR} and the references therein) rely
exactly on this approach. However, in our case we will not obtain a true
Lyapunov function on the unbounded domain $\Omega = \R^1$, because of the fact
that $\bfU=(U,V)^\top$ is not a true steady state of our scaled system
\eqref{eq:RDS.vector.scaled}, see the Lagrange multiplier $\Lambda$ in the
profile equation \eqref{eq:I.ProfEqn}. Only in the very special case
$\alpha = \beta $ and $d_1=d_2$ one has $\Lambda\equiv 0$, and we easily will
obtain
$\calE_\rmB(\bfu(\tau)|\bfU) \leq \ee^{-\frac12 \tau}
\calE_\rmB(\bfu(0)|\bfU)$.

In Section \ref{se:LinDiffEq} we explain our method by applying it to the
scaled linear diffusion equation $u_\tau = D u_{yy} + \frac 
y2 u_y$ with boundary conditions $u(\tau,\pm\infty)=A_\pm>0$. This leads to
the energy-dissipation estimate 
\[
\frac{\rm d}{\rm d \tau}\calE_\phi(u|U) = -\IFisher(u) - \frac12 \calE_\phi(u|U) \quad
\text{with } \IFisher(u)= D \int_\R  \phi''(u/U) \big( (u/U)_y\big)^2
U \dd y. 
\] 
The new and very helpful term $-\frac12\calE_\phi(u \,| \,U)$ arises from the drift
term $\frac y2 u_y$ which stems from the parabolic scaling. Using
$\calI_\text{Fisher}\geq 0$, we obtain $\calE_\phi(u(\tau)| \, U) \leq
\ee^{-\frac12 \tau} \calE_\phi(u(0)| \, U)$ without using any Poincar\'e or
log-Sobolev estimate on $\R$. For that reason, the factor $1/2$ will be called 
the \textit{bonus factor}, subsequently.

In Section \ref{se:RDS} we show that in our case we have $\frac\rmd{\rmd
  \tau} \calE_\rmB(\bfu(\tau)|\bfU) = - \calD_\rmB(\bfu)$ with a dissipation
functional that can be written as
\[
\calD_\rmB = \IFisher + \ee^\tau \, \Dreact + \frac12 \calE_{\rm B}  - \Imix \quad
\text{with }\Dreact= \int_\R k U^\alpha \,\Gamma\big( \big(\frac{u(y)}{U(y)}\big)^\alpha,
\big(\frac{v(y)}{V(y)}\big) ^\beta \big) \dd y,
\]
where $\Gamma(a,b)\geq 0$ is defined by
\begin{equation} \label{eq:defGamma}
\Gamma (a,b) := \begin{cases} (a{-}b) \big( \log a {-} \log b \big) \geq 0 &
  \text{for }a,b>0,\\ 
 0 & \text{for } a=b=0, \\
\infty& \text{for }(0,c) \text{ and }(c,0) \text{ for }c>0.
\end{cases}
\end{equation}
Here $\IFisher$ consists of two non-negative terms, one for $u$ and one for
$v$. The special form of $\Dreact$ and its positivity arise from the special
interaction of the mass-action law and the Boltzmann entropy, namely with 
$\rmD\calE_\rmB(\bfu|\bfU)=\binom{\log u}{\log v}$ and the logarithm rules one
finds $(u^\alpha{-}v^\beta) \tbinom{\alpha}{-\beta} \cdot \tbinom{\log u}{\log
  v} = \Gamma ( u^\alpha , v^\beta )\geq 0$. 
Again we have the bonus factor $1/2$ arising from the drift term $\frac
y2 \bfu_y$. The new and difficult term is the mixed term
\[
\Imix(\bfu) = \int_\R \Big(  \big( 1- \frac
uU \big)\alpha -  \big(1- \frac vV\big)\beta  \Big) \Lambda \dd y , 
\]
which arises from the fact that $\bfU$ is not a steady state, but needs the Lagrange
multiplier $\Lambda$, see the profile equation \eqref{eq:I.ProfEqn}. In
particular, $\Imix$ does not have a specific sign.     

The derivation of the useful splitting of $\calD_\rmB$ is part of Section
\ref{su:Dissipation} and the precise statement can be found in Proposition
\ref{prop:Dissipation}.  Because of the unboundedness of $\Omega =\R^1$ we will
not be able to take advantage of the Fisher information $\IFisher$, but we can
rely on the bonus factor $1/2$. Moreover, the reactive dissipation $\Dreact$
has the factor $\ee^\tau$ in front, from which we will benefit in Sections
\ref{se:SpecialCase} and \ref{se:GeneralCase} to control $\Imix$.  Since the
mixed term $\Imix$ has no fixed sign, it is possible that the relative
Boltzmann entropy $\calE_\rmB$ may grow, i.e.\ it is not a true Lyapunov
function. So our aim will be to show that
\begin{equation}
  \label{eq:I.ImixControl}
  \frac12 \,\calE_\rmB(\bfu\,|\,\bfU) +\ee^\tau \Dreact(\bfu) - \Imix(\bfu)
\geq (\eta {-} \mu \, \ee^{-\tau} ) \,\calE_\rmB(\bfu\,|\,\bfU) - K \ee^{-\sigma \tau} 
\quad \text{ for
}\tau\geq 0
\end{equation}
for some $\eta,\, \sigma >0$ \AAA and $ \mu,\, K \geq 0$. From this, \EEE our convergence results in
Theorems \ref{thm:specialCase} and \ref{thm:generalCase} will follow.

The control of the problematic term $\Imix$ is different in the simpler case
$\alpha = \beta  \geq 1$ (see Section \ref{se:SpecialCase}) and in the more
difficult case $\alpha > \beta \geq 1$ (see Section
\ref{se:GeneralCase}). For $ \alpha = \beta \geq 1$, the integrand of $\Imix $
only depends on $\frac uU -\frac vV$ and thus can be controlled by $\Dreact$ alone. 
Hence, we obtain \eqref{eq:I.ImixControl} with $\eta=1/2$ and 
$\sigma = 1$ if $\alpha = \beta  \in [1,2)$ or $\sigma = 1/(\alpha{-}1)$ for 
$\alpha = \beta \geq 2$, 
without any restriction on the self-similarity profile $\bfU$.  
Using a version of Gr\"onwall's inequality (see Lemma \ref{le:Gronwall}), we find
for $ \alpha = \beta \in [1,3)$ the decay
$\calE_\rmB(\bfu(\tau) \,|\,\bfU) \leq C \ee^{- \min \{\eta,\sigma \}\tau} 
 =  C \ee^{-\tau/2}$,  whereas for
$ \alpha = \beta > 3$ we have a slower decay like $\ee^{-\tau/(\alpha-1)}$.

For $ \alpha > \beta \geq 1$ it is more difficult to control $\Imix$ and we
need to exploit the term $ \frac12 \,\calE_\rmB$ with the bonus factor. From
\cite{MieSch21?ESPS} we know that for small $|A_+{-}A_-|$ also
$\|\Lambda\|_\infty$ is small. Thus, for sufficiently small $|A_+{-}A_-|$, we
have
\[
\theta := (\alpha {-}\beta) \sup\bigset{\Lambda(y)/V(y) }{ y \in \R} <1/2,
\]
and Theorem \ref{thm:generalCase} shows that \eqref{eq:I.ImixControl} holds for
$\eta=1/2-\theta>0$ and suitable $K$ and $\sigma$. It remains open whether
in the case $ \alpha > \beta $ the asymptotic profiles with large difference
$|A_+{-}A_-|$ are stable or not. \AAA We remark that a flatness condition for the
profile $\bfU$ (which is encoded in $\|\Lambda\|_\infty \leq C |A_+{-}A_-|$)
appears also in \cite[Thm.\,1.1]{MarPan01DPSC}.  

We expect that our approach based on energy-dissipation estimates is flexible
enough to allow for several generalizations. Based on the vector-valued
existence results for similarity profiles in \cite{MieSch21?ESPS}, it should be
possible to treat general reaction systems for $i_*$ species interacting via
$r_*$ reaction pairs with mass-action kinetics, where $m_*:=i_*-r_*\geq 1$ provides
the dimension of the equilibrium manifold which will then include the
similarity profile $\bfU$. Of course, the problem of controlling the mixed term
$\Imix$ will be more involved, because $\Imix$ now involves $m_*$ Lagrange multipliers. 
Moreover, our convergence theory works equally well for space
dimension $d \geq 2$: as soon as the existence of similarity profiles is
established, the energy-dissipation estimates can be done with a bonus factor~$d/2$.  \EEE

\section{Convergence to self-similarity for the linear diffusion equation 
on the whole space}
\label{se:LinDiffEq}

In this section, we demonstrate our proceeding to the well-studied linear
diffusion equation
\begin{equation}
	\label{eq:LinDiff}
\tilde{u}_t = D \, \tilde \Delta \tilde{u} \; \text{ on } \; \R^d,
\end{equation}
with diffusion constant $D >0$.  For initial data $\tilde{u}^0 \in L^1(\R^d)$, it
is already known that the solutions behave asymptotically like a Gaussian, see
e.g. \cite[Sec.\,2.4]{Jung16EMDP}. In this paper, we are interested in the
long-time behavior of solutions which have nontrivial boundary conditions for
$|x|\to \infty$, such that the solutions have infinite mass. 

In the one-dimensional case, we consider the
diffusion equation \eqref{eq:LinDiff} together with the boundary conditions 
\[
\tilde u(t,\pm \infty):=\lim_{x\to \pm \infty}\tilde u(t,x) = A_\pm 
\]
and ask how the solution mixes these two steady states
$A_\pm$ when time $t$ goes to $\infty$.
Because of the linearity, it is not difficult to prove that for every given
pair $(A_-,A_+) \in \R^2$ of asymptotic boundary conditions, the
solution converges uniformly in $x \in \R$ to the following self-similar
solution
\begin{align}
\label{eq:LinDiff.profile}
U(x/ \sqrt{t}) := \frac12 (A_+ {-} A_-) \; \mathrm{erf}\big( x / \sqrt{4Dt} \big) 
+ \frac12 (A_+ {+} A_-),
\end{align}
where $\mathrm{erf}(x):= \frac{2}{\sqrt{\pi}} \int_0^x \exp(-z^2) \dd z$ is the
error function.  However, we do not want to use the linearity to 
verify this convergence, neither the exact representation of the profile given
by the error function, with the idea in mind to generalize the following
strategy to the given nonlinear reaction-diffusion system.  Hence, we will use
entropy estimates to prove this convergence.

Before doing so, we look at the profile function \eqref{eq:LinDiff.profile}
from a different perspective. We see that $U$ depends on the quotient
$x/ \sqrt{t}$ instead of the variables $t$ and $x$ separately. This motivates
to do a transformation into the so-called parabolic scaling variables
given by $y= (1{+}t)^{-1/2} x \in \R^d$ and $\tau=\log(1{+}t)$.  Returning to the
multi-dimensional case, we define
\[
u(\tau,y):=\tilde{u}(t,x) = \tilde{u}\big(\ee^\tau {-}1, \,\ee^{\tau/2} y\big)
\]
and find the scaled diffusion equation with the same asymptotic boundary conditions:
\begin{equation}
	\label{eq:transLinDiff}
u_\tau = D \, \Delta u +  \frac{1}{2} y\cdot \nabla  u  \quad \text{and} \quad
u(\tau,y)-U(y) \to 0 \text{ for }|y|\to \infty, 
\end{equation}
where $\Delta$ and $\nabla$ are now taken with respect to $y \in \R^d$. The
asymptotic boundary conditions are given by a fixed function $U:\R^d \to \R$,
which we take as a self-similar profile, i.e.\ it satisfies the profile equation 
\begin{equation} 
	\label{eq:LinDiff.steady-state}
D \,\Delta  U + \frac{1}{2} y \cdot \nabla U =0 \ \text{ on } \R^d .
\end{equation}
Clearly, the solutions in \eqref{eq:LinDiff.profile} provide all possible
solutions for the case $d=1$. For $d \geq 2$ the set of solutions is much
richer, even when restricting to the case $U\in \rmC^2(\R^d)$ with $\inf U \geq
\underline U>0$.  Of course, we again see that $\tilde u(t,x)=
U\big((1{+}t)^{-1/2}x\big)$ is an exact self-similar solution of the unscaled
equation \eqref{eq:LinDiff}. 

To prepare for the subsequent analysis for reaction-diffusion system, we
now show convergence of all solutions of the scaled linear diffusion
equation in the sense that the relative entropy
\begin{align*}
\calE_\phi(u \vert U) &:= \int_\R \phi(u / U)\, U \dd y = \int_\R  \phi(\rho)
U \dd y  , \quad \text{where } \rho=u/U, 
\end{align*}
converges exponentially to $0$. Here $\phi$ is an arbitrary convex entropy
function fulfilling $\phi(\rho)\geq \phi(1) = 0$ for all $\rho\geq
0$. We call the arising exponential decay rate $d/2$ the bonus factor, because
it solely comes from the scaling, i.e.\ from the drift term $\frac12 y\cdot \nabla
u$. 

\begin{proposition}[Decay in the linear diffusion equation]
\label{pr:LinDiffEqn}
Consider the scaled linear diffusion equation and let $U \in \rmC^2_\rmb(\R^d)$ be
the similarity profile satisfying \eqref{eq:LinDiff.steady-state} and $U(y)\geq
\underline U>0$.
Then, all solutions $u$ of the Cauchy problem \eqref{eq:transLinDiff} fulfilling
$\calE_\phi(u^0 \vert U) = \int_{\R^d} \phi(u^0 / U) U \dd y < \infty$
converge to $U$ in the sense that 
\begin{align*}
\calE_\phi (u(\tau) \vert U) 
\leq \ee^{-d\tau/2} \,\calE_\phi (u^0 \vert U) \ \text{ for all } \tau >0.
\end{align*}
\end{proposition}
\begin{proof}
To simplify the calculation we use the relative density 
$\rho(\tau, y) := u(\tau,y)/U(y)$ and observe that the scaled diffusion
equation \eqref{eq:transLinDiff} takes the form 
\begin{align*}
U \rho_\tau &= D \big( U \Delta \rho + 2 \nabla U \cdot \nabla\rho 
	+ \rho \Delta U\big)+ \frac{1}{2} y\cdot \big(U \nabla \rho 
	{+} \rho \nabla U \big)\\
& = D \big( U \Delta\rho 
	+ 2 \nabla U\cdot \nabla\rho\big) + \frac{1}{2} U \,y\cdot\nabla \rho,
\end{align*}
where the last identity follows by inserting the profile equation
\eqref{eq:LinDiff.steady-state} for $U$. 

We compute 
the time derivative of the relative entropy. It holds
\begin{align*}
\frac{\rmd}{\rmd \tau} \calE_\phi (u(\tau) \vert U) 
&=  \int_\Rd \phi'(\rho) \rho_\tau U \dd y 
=  \int_\Rd  \phi'(\rho)\big\lbrace D \big(U \Delta \rho 
	+ 2\nabla U \cdot \nabla  \rho \big)+ \frac{1}{2}U y\cdot\nabla\rho 
  \big\rbrace  \dd y \\
&\overset*=  \int_\Rd \big\lbrace -D \phi''(\rho)|\nabla \rho|^2 U 
     + \big( D\nabla U +\frac12 U y\big) \cdot\big(\phi'(\rho)\nabla \rho\big) 
      \big\rbrace \dd y  
\\ 
&\overset*= \int_\Rd \big\lbrace -D \phi''(\rho)|\nabla \rho|^2 U  - 
 \big( D\Delta U +\frac12\,y\cdot \nabla U + \frac12(\DIV y) U \big) 
  \phi(\rho)\big\rbrace \dd y  
\\
& = - \IFisher(\rho) \  - 0 \ - \frac d2 \int_\Rd
\phi(\rho) U \dd y \ =:\  - \calD_\phi(\rho). 
\end{align*}
Here $\overset*=$ indicates an integration by parts where we use $\rho(y) \to 1$
and $\phi'(1)=0$. For the second last identity we used the Fisher information 
\[
\IFisher (\rho) := D \int_\Rd \phi''(\rho)|\nabla \rho|^2 U \dd y \geq 0 
\]
and the profile equation $D\Delta U+ \frac12\,y\cdot \nabla U=0$ once
again. The bonus factor arises from $\frac12 \DIV y = d/2$.

Thus, the dissipation $\calD_\phi$ is non-negative and
satisfies $\calD_\phi(\rho)\geq \frac d2\:\calE_\phi(u \vert U)$ yielding 
\begin{align*}
\frac{\d}{\d \tau} \calE_\phi (u(\tau) \vert U) 
&=  - \calD_\phi(\rho) \leq -\frac d2\:\calE_\phi(u(\tau) \vert U) .
\end{align*}
By Gr\"onwall's Lemma, we obtain 
exponential convergence in $\tau$. More precisely, we have
\begin{align*}
\calE_\phi  (u(\tau) \vert U) \leq \ee^{- d\tau/2} \calE_\phi  (u(0) \vert U) 
\text{ for } \tau >0,
\end{align*} 
as it was claimed.
\end{proof}

In the above proof, we see the essential benefit of the parabolic scaling. The extra
term $\frac{1}{2} y \cdot \nabla u$ featuring in the scaled diffusion equation
\eqref{eq:transLinDiff} leads to the so-called bonus factor $d/2$ in the
differential inequality for the relative entropy, which in turn provides
convergence and an explicit decay rate.

\section{The reaction-diffusion system}
\label{se:RDS}

The first part of this section is dedicated to introduce the coupled reaction-diffusion 
system \eqref{eq:RDS.vector} and the scaled one \eqref{eq:RDS.vector.scaled} 
together with its important properties in more detail. 
Then in Section \ref{su:Dissipation}, we derive the dissipation functional for the scaled 
system and prove an appropriate splitting of it.

\subsection{The system and its similarity profile}
\label{su:RDSandProfile}

Consider a coupled system of two nonlinear reaction-diffusion 
equations on the unbounded domain $\Omega = \R^1$ which present the concentration 
change of the diffusing species $X_1$ and $X_2$ interacting through the single reversible 
reaction $\alpha X_1 \rightleftharpoons \beta X_2$ with each other. When we denote their 
densities with $\tilde{u},\tilde{v} \geq 0$, respectively, the mass-action law leads to the 
system
\begin{equation} \label{eq:RDS}
\begin{aligned}
\tilde{u}_t &= d_1 \tilde{u}_{xx} + \alpha k (\tilde{v}^\beta {-} \tilde{u}^\alpha), \\
\tilde{v}_t &= d_2 \tilde{v}_{xx} -\beta k (\tilde{v}^\beta {-} \tilde{u}^\alpha),
\end{aligned}
\end{equation}
for $t > 0$ and $x \in  \R^1$, where the diffusion constants $d_1,d_2$ and the 
reaction rate $k$ are assumed to be positive. The set of constant steady states is a 
one-parameter family given by
\[
\bigset{ (A^\beta, A^\alpha) }{ A >0 }.
\]
We are interested in the behavior of solutions where the initial data $(u^0,v^0)$ is in 
equilibria at infinity, i.e. where for two given constants $A_-, A_+ >0$ the continuous 
initial data satisfies the asymptotic boundary conditions
\[
\big(\tilde{u}^0(\pm \infty), \tilde{v}^0(\pm \infty)\big) := \lim_{x\to \infty}
\big(\tilde{u}^0(\pm x), \tilde{v}^0(\pm x) \big)  = \big(A_\pm^\beta,
A_\pm^\alpha \big). 
\]
Motivated by Section \ref{se:LinDiffEq}, we transform the system \eqref{eq:RDS} into
 parabolic scaling coordinates
\[
y=x/\sqrt{t{+}1} \quad \text{and} \quad \tau=\log(t{+}1).
\]
Then the transformed system reads
\begin{equation} \label{eq:transRDS}
\begin{aligned}
u_\tau &= d_1 u_{yy} + \frac{y}{2} u_y+ \ee^\tau \alpha k (v^\beta {-} u^\alpha), \\
v_\tau &= d_2 v_{yy} + \frac{y}{2}v_y - \ee^\tau \beta k (v^\beta {-} u^\alpha). 
\end{aligned}
\end{equation}
Accordingly, the continuous initial data $(u^0,v^0)$ satisfies the asymptotic
boundary conditions
\begin{equation} \label{eq:transRDS.BC}
\big( u^0(\pm \infty), v^0(\pm \infty) \big) = \big(A_\pm^\beta, A_\pm^\alpha \big).
\end{equation}
Note the exponential factor that appears in front of the reaction terms in
\eqref{eq:transRDS} as the reaction does not transform like the parabolic
terms.  At a first glance, one might say that the transformed system looks much
more complicated than the original one since it is now non-autonomous.  On top
of that, the factor is exponentially growing in time, which could impair
convergence.  On further consideration, however, we will see that the prefactor
$\ee^\tau$ is beneficial from a technical point of view and makes things work
in the end.  Luckily, the reaction term comes with a difference; thus, the
prefactor indicates how the solutions probably behave for large times.  To
prove rigorously that this is true is the aim of Sections \ref{se:SpecialCase}
and \ref{se:GeneralCase}. 

But already now we can imagine that the exponentially
growing factor forces the reaction to equilibrate for $\tau\to
\infty$. However, there might still be nontrivial reaction fluxes
$q=k \ee^\tau (v^\beta- u^\alpha)$ for $\tau\to \infty$, which can be seen as
the limit of the type ``$\infty\cdot 0$''. As discussed in
\cite{MieSch21?ESPS, MieSch23?SSPC}, the similarity profile $y\mapsto
\bfU(y)=(U(y),V(y))^\top$ has to satisfy the
following differential-algebraic system
\begin{align}
	\label{eq:SelfSimProfile}\hspace*{-0.5em}
\binom00&=\binom{d_1 U''(y) {+} \frac y2\,U'(y)}{d_2 V''(y) {+} 
 \frac y2\,V'(y)}  {+} \Lambda(y) \binom{\alpha}{\!-\beta\!} , 
 \ \ U(y)^\alpha= V(y)^\beta,  \ \ 
\binom{U(\pm\infty)}{V(\pm\infty)} = \binom{A_\pm^\beta}{A_\pm^\alpha}. 
\end{align}
It is possible to eliminate $\Lambda$ and $V= U^{\alpha/ \beta}$ algebraically
to obtain a nonlinear ODE for $U$ alone, namely
\begin{equation*} 
\big( \beta d_1 U + \alpha d_2 U^{\alpha/{\beta}} \big)'' 
	+ \frac{y}{2}  \big( \beta U + \alpha U^{\alpha/ \beta} \big)' = 0,  
	\quad \text{ with } U(\pm \infty) = A_\pm^\beta. 
\end{equation*}
In \cite{MieSch21?ESPS} it is shown that for all $(A_-,A_+)$ there exists
a unique solution $\bfU$ of \eqref{eq:SelfSimProfile}.

We call the functions $\bfU = (U,V)^\top$ \textit{similarity profiles} and aim to
prove that solutions $\bfu = (u,v)^\top$ to
\eqref{eq:transRDS}--\eqref{eq:transRDS.BC} converge towards the profiles in the
sense that the relative Boltzmann entropy $\EB $ satisfies the qualitative
estimate
\begin{equation}
\label{eq:EntropyConvergence}
\EB (\bfu(\tau)\, \vert \, \bfU ) \leq \tilde C \, \ee^{-\eta \tau} \, 
 \EB ( \bfu(0) \, \vert \, \bfU) + \tilde K\, \ee^{-\sigma \tau}, 
\end{equation}
where the rates $\eta,\ \sigma >0$ and the constants $\tilde C,\ \tilde K$
depend only on the given problem data, but not on the initial condition
$\bfu(0)$. This then implies exponential convergence of
$\EB(\bfu(\tau)\, \vert \, \bfU )$ with exponential rate
$\min\{\eta,\sigma\}>0$.
% Since in the case of an unbounded domain the spectrum will be continuous containing 
% $0$, we expect to have algebraic decay properties, which means exponential
% convergence with respect to the time variable $\tau := \log(t {+}1)$.

In \cite[Lem.\,3.2]{MieSch21?ESPS} it is additionally shown that the profiles
$\bfU =(U,V)^\top$ solving \eqref{eq:SelfSimProfile} are monotone, i.e.\ for $A_-<A_+$
one has $U'(y),V'(y)>0$ for all $y \in\R$.  As a consequence, we have
$U(y)> A_-^\beta>0$ and $V(y) > A_-^\alpha>0$ for all $y \in \R$ so that the
relative entropy, where we have the relative densities $\rho=u / U$ and
$\zeta= v / V$ in the argument of the Boltzmann function, is well-defined.

\subsection{Suitable split of dissipation}
\label{su:Dissipation}

Let us recall that the usual procedure is to take a relative entropy $\calE_\phi$ and 
to show that it fulfills for all times the so-called entropy-dissipation relation
\[
\frac{\d}{\d \tau} \calE_\phi (\bfu \, \vert \, \bfU) 
=- \calD_\phi (\bfu) \leq 0
\]
for a \textit{non-negative} dissipation functional $\calD_\phi$. 
In our case, we cannot expect the monotonicity of the
mapping $\tau \mapsto  \calE_\phi (\bfu (\tau) \, \vert \, \bfU) $ as it is
posed on the whole space and $\bfU$ is not a true steady state. This is in
contrast to \cite{DesFel06EDTE, DesFel07EMRD, BLMV14LFBD, Miel17UEDR}, 
where the unscaled system \eqref{eq:RDS} is
studied on bounded domains and where exact steady states exist. 
However, in Sections \ref{se:SpecialCase} and \ref{se:GeneralCase} we will prove
that the entropy-dissipation relation is correct \AAA up to exponentially decaying
terms, see \eqref{eq:I.ImixControl} or Lemma \ref{le:Gronwall}. \EEE 

Let us take the relative Boltzmann entropy
\begin{equation*}
  \EB( \bfu \, \vert \, \bfU) 
  =  \int_\R \big( \LB(\rho) U +  \LB(\zeta) V\big) \dd y \quad 
  \text{ where } \rho := \frac{u}{U} \text{ and }  \zeta := \frac{v}{V},
\end{equation*}
as it goes hand in hand with the mass-action kinetics.

The aim of this section is first to derive the dissipation functional 
$\calD_\rmB$ that fulfills
\begin{equation}
 \label{eq:EntDissLaw}
\frac{\d}{\d \tau}  \EB (\bfu \, \vert \, \bfU)
=:- \calD_{\rmB} (\rho,\zeta),
\end{equation}
and second to find a suitable partition of it in good and problematic terms,
which is useful since $\calD_\rmB$ has -- as we already suspect -- no fixed sign in
our setting.  Thus, we will examine the terms of which it consists in an
appropriate way.  Note that we will write the dissipation terms as functions of
the relative densities $\rho = u/ U$ and $\zeta = v/ V$, whereas we keep the
relative entropy in standard form in terms of $\bfu =(u,v)^{\top}$.

\begin{proposition}
\label{prop:Dissipation}
The dissipation $\calD_{\rmB}$ fulfilling \eqref{eq:EntDissLaw} can be
decomposed as 
\[
\calD_{\rmB} (\rho,\zeta) =   \IFisher (\rho,\zeta) + \frac12\,  \EB(\bfu\, | \,
\bfU) -\Imix (\rho,\zeta)  + \ee^\tau\, \calD_\mathrm{react} (\rho,\zeta),
\]
where
$\calD_\mathrm{react}(\rho,\zeta) := \int_\R k U^\alpha \Gamma(\rho^\alpha,
\zeta^\beta) \dd y \geq 0 $\vspace{0.15em} is the reactive dissipation and
$ \IFisher (\rho, \zeta) := \int_\R d_1 U \LB''(\rho) \rho_y^2 + d_2 V
\LB''(\zeta) \zeta_y^2 \dd y \geq 0$\vspace{0.15em} is known as the Fisher
information. The bonus term $\tfrac12 \EB$ stems from the transport term
$\frac y2 \pl_y \bfu$, and the remaining term
\begin{equation}
  \label{eq:Imix.general}
  \Imix(\rho,\zeta) := 
\int_\R \big(  (1 {-} \rho)\alpha - (1 {-} \zeta)\beta\big) \Lambda \dd y,
\end{equation}
arises because of the Lagrange multiplier $\Lambda$ which features in the
profile equation \eqref{eq:SelfSimProfile}. This term
is called the \emph{mixed term},  because it is the only addend without sign.
\end{proposition}
\begin{proof}
Take the relative Boltzmann entropy $ \EB$ given by the functional
\[
 \EB( \bfu \, \vert \, \bfU) =   \int_\R U  \LB(\rho) 
+ V \LB(\zeta) \dd y.
\] 
The relative densities $\rho(\tau, y) := u(\tau,y)/ U(y)$ and 
$\zeta(\tau, y):= v(\tau,y)/V(y)$ satisfy
\begin{align*}
U \rho_\tau &= d_1 \big( U \rho_{yy} + 2 U' \rho_y 
	+  U'' \rho \big)+ \frac{y}{2} \big(U \rho_y 
	+ U' \rho \big) 
	+ \alpha k\ee^\tau U^{\alpha} (\zeta^\beta{-}\rho^\alpha), \\
V \zeta_\tau &= d_2 \big( V \zeta_{yy} + 2 V' \zeta_y 
	+ V'' \zeta \big)
	+ \frac{y}{2} \big(V \zeta_y +   V' \zeta \big) 
	-  \beta k\ee^\tau V^{\beta} (\zeta^\beta{-}\rho^\alpha).
\end{align*}
Thus, computing the time derivative of the relative entropy yields
\begin{align*}
\frac{\d}{\d \tau}  \EB( \bfu \, \vert \, \bfU)
&=  \int_\R U \LB'(\rho) \rho_\tau 
	+ V \LB'(\zeta) \zeta_\tau  \dd y \\
&=  \int_\R \LB'(\rho)\Big\lbrace d_1 \big(U \rho_{yy} 
	+ 2 U' \rho_y + U''\rho \big)
	+ \frac{y}{2} \big(U \rho_y + U' \rho \big) \Big\rbrace \dd y \\
& \quad + \int_\R \LB'(\zeta)\Big\lbrace d_2 \big( V \zeta_{yy} 
	+ 2 V' \zeta_y +  V'' \zeta \big)
	+ \frac{y}{2} \big(V \zeta_y + V' \zeta \big) \Big\rbrace  \dd y \\
& \quad -  \ee^\tau  \int_\R kU^\alpha \big( \beta \LB'(\zeta) 
	- \alpha \LB'(\rho)  \big) (\zeta^\beta{-}\rho^\alpha) \dd y \\
&=: -\calD_\mathrm{diff}(\rho, \zeta) - \ee^\tau \calD_\mathrm{react}(\rho, \zeta),
\end{align*}
where we used the relation $U^\alpha = V^\beta$ to simplify the reaction terms. 
Let us consider the reactive dissipation first.  We can use the logarithmic identities to 
obtain a sign for $\calD_\mathrm{react}$. It holds
\begin{align*}
\calD_\mathrm{react}(\rho,\zeta) 
&:=  \int_\R k U^\alpha \big( \beta \LB'(\zeta) 
	- \alpha \LB'(\rho) \big) (\zeta^\beta{-}\rho^\alpha) \dd y 
=\int_\R k U^\alpha \Gamma(\rho^\alpha,\zeta^\beta) \dd y \geq 0,
\end{align*}
where $\Gamma$ is defined in \eqref{eq:defGamma}. 
Next, we explore the remaining diffusive dissipation $\calD_\mathrm{diff}$.  
We re-sort and obtain
\begin{align*}
\calD_\mathrm{diff}(\rho, \zeta) 
&=- \int_\R d_1 U \LB'(\rho) \rho_{yy} 
	+ d_2 V \LB'(\zeta) \zeta_{yy} \dd y \\
& \quad -  \int_\R  \LB'(\rho) \rho \big( d_1 U'' 
	+ \frac{y}{2} U' \big) +  \LB'(\zeta) \zeta \big( d_2 V'' 
	+ \frac{y}{2} V' \big) \dd y \\
& \quad-  \int_\R  \LB'(\rho) \rho_y \big( 2 d_1 U' 
	+\frac{y}{2} U  \big) +  \LB'(\zeta) \zeta_y \big( 2 d_2 V' 
	+\frac{y}{2} V \big) \dd y .
\end{align*}
In the same manner as for the scaled diffusion equation, the idea is to integrate by 
parts twice.  For the boundary terms, we use the limits $\rho(y) \to 1$ and 
$\zeta(y) \to 1$ for $y \to \pm \infty$ and the property $\LB'(1) = 0$. 
The first integral addend leads to the Fisher information 
\begin{align*}
\IFisher (\rho,\zeta)
:= \int_\R d_1 U \LB''(\rho) \rho_y^2 
	+ d_2 V \LB''(\zeta) \zeta_y^2 \dd y \geq 0.
\end{align*}
Hence, we obtain
\begin{align*}
 \calD_\mathrm{diff}(\rho, \zeta) 
&= \IFisher(\rho,\zeta)  - \int_\R  \LB'(\rho) \rho 
	\big( d_1 U'' 
	+ \frac{y}{2} U' \big) +  \LB'(\zeta) \zeta \big( d_2 V'' 
	+ \frac{y}{2} V' \big) \dd y \\
& \quad -  \int_\R  \LB'(\rho) \rho_y \big( (2 d_1{-}d_1) U' 
	+\frac{y}{2} U  \big) 
	+  \LB'(\zeta) \zeta_y \big( (2 d_2{-}d_2) V' +\frac{y}{2} V \big).
\end{align*}
In the last line, we see the total derivatives of $\LB(\rho)$ and 
$\LB(\zeta)$, respectively. Thus, integration by parts of these integral 
terms yields the factors $d_1 U''+\frac y2U' + \frac12 U= -\alpha \Lambda
+ \frac12 U$ and $d_2 V''+\frac y2 V' + \frac12 V= \beta \Lambda + \frac12 V $,
respectively, where we used the profile equation
\eqref{eq:SelfSimProfile}. Using additionally  $ \LB'(\rho) \rho {-} \LB(\rho)
  = \rho{-} 1$ we arrive at  
\begin{align*}
\calD_\mathrm{diff} (\rho, \zeta) &=  \IFisher(\rho, \zeta) 
	 + \frac12 \, \EB( \bfu \, \vert \, \bfU)- 
       \int_\R \big( (1 {-} \rho) \alpha \Lambda -(1 {-} \zeta) \beta \Lambda \big) \dd y \\
&=  \IFisher(\rho, \zeta)  + \frac12\,  \EB( \bfu \, \vert \, \bfU)
	 - \Imix(\rho,\zeta),
\end{align*}
which verifies the desired decomposition.
\end{proof}

In the proof of Proposition \ref{prop:Dissipation}, we saw that due to the
mass-action kinetics, the relative Boltzmann function $\phi = \LB$ is the
only choice for the given reaction-diffusion system if $\alpha \neq \beta$ in order
to obtain a sign for the reactive dissipation $\Dreact$.
However, if $\alpha = \beta$, also other entropy functions can be chosen.
A common family of entropy functions is given by
\begin{align}
\label{eq:F.p}
F_p(z) := {\footnotesize
   \begin{cases} 
	\frac{1}{p (p{-}1)} \big( z^p - p z+ p -1\big) 
		&\text{for } p \in \mathbb{R} \setminus \{0,1 \},\\
	z \log z -z+1 &\text{for } p=1, \\
	z-\log z-1  &\text{for } p =0,
   \end{cases}}
\end{align}
which is determined by the conditions $F_p''(z) = z^{p{-}2}$ and
$F_p(1) = F_p'(1)=0$. Further, it satisfies the following lower bounds:
\begin{subequations}
\begin{align}
  \label{eq:Fp.equiv}
&\text{For all } p \in (0,1) \text{ and } z >0: \quad 
&&F_p(z) \geq \frac1p F_1(z)= \frac1p \LB(z),\\
\label{eq:Fp.bound}
&\text{for all } p >0 \text{ and } z >0: \quad 
&&F_p(z) \geq \frac{1/2} {\max\{p,1{-}p\}}\, F_{1/2}(z),
\end{align}
\end{subequations}
see \cite[Eqn.\,(3.2)]{MieMit18CEER}.  In fact, using this family of
entropies in the case $\alpha = \beta$ leads to improved estimates as we will
see in Section \ref{su:al=be.GenEntropy}.  But also in Section
\ref{se:GeneralCase}, where for $\alpha \neq \beta$ the convergence of the
relative Boltzmann entropy is studied, the family of entropy functions
\eqref{eq:F.p} is used, but in this case only for technical reasons during the
estimates.

We can define the relative entropy associated to the function $F_p$ by
\[
\calE_p (\bfu (\tau)\, \vert \, \bfU) := \int_\R U(y) F_p(\rho(y)) + V(y)
F_p(\zeta(y)) \dd y  
\quad \text{where } \rho := u/U \text{ and } \zeta := v/V,
\]
such that $\calE_1=\EB$.  The entropy $\calE_{1/2}$ is special
because $F_{1/2}(u/U)U=2 \big( \sqrt{u}-\sqrt U\big)^2$. Hence we have
\[
\int_\R F_{1/2}(u/U)U\dd y = 2 \big\|\sqrt u - \sqrt U\|_{\rmL^2}^2 =:2
\HELL(u,U)^2,
\]
where $\HELL$ denotes the Hellinger distance between two (densities of)
non-negative measures. Using \eqref{eq:Fp.equiv} we see that the Hellinger
distance between $\bfu=(u,v)^\top$ and $\bfU=(U,V)^\top$ can be controlled by
 $\calE_p$ for all $p>0$. Indeed we have
\begin{align}
  \label{eq:Helli.calEp}
  \HELL(u,U)^2+\HELL(v,V)^2 &=\frac12
  \calE_{1/2}(\bfu\,|\,\bfU) \leq \max\{p,1{-}p\}  \calE_p(\bfu\,|\,\bfU).
\end{align}
\EEE

As last part of this section, we will derive the corresponding dissipation functional 
$\calD_p$ which fulfills
\[
\frac{\rmd}{ \rmd \tau} \calE_p (\bfu (\tau)\, \vert \, \bfU)  = - \calD_p(\rho,\zeta)
\]
and clarify the terms of which it consists. Notice that the following is only true
if the stoichiometric coefficients coincide.
\begin{proposition}
\label{prop:Diss.gen.Entr}
Let $\alpha = \beta$ and $p \not\in \{0,1\}$. The dissipation functional $\calD_{p}$ 
fulfilling the above can be written as
\[
\calD_{p} (\rho,\zeta) =   \IpFisher (\rho,\zeta) + \frac12\,  \calE_p(\bfu\, | \,
\bfU) -\Ipmix (\rho,\zeta)  + \ee^\tau \Dpreact (\rho,\zeta),
\]
where $\Dpreact (\rho,\zeta) 
:= \int_\R k  U^\alpha  \frac{\alpha}{p{-}1}(\zeta^{p{-}1} {-} \rho^{p{-}1})
( \zeta^\alpha {-} \rho^\alpha) \dd y \geq 0$ is the reactive dissipation, 
the Fisher information takes the form $ \IpFisher (\rho, \zeta)
 :=  \int_\R d_1 U \rho^{p{-}2} \rho_y^2 
	+ d_2 V \zeta^{p{-}2} \zeta_y^2 \dd y \geq 0$, and the mixed term, given by
\[
 \Ipmix(\rho,\zeta) := 
\int_\R \frac1p \big(  \zeta^p {-} \rho^p \big)\, \alpha \Lambda \dd y
%\quad \text{ for } \Lambda(y)= \frac{d_2{-}d_1}{2\alpha}U''(y),
\]
is again the only addend without sign.
\end{proposition}
\begin{proof}
Following the steps of the proof of Proposition \ref{prop:Dissipation}
and using $\alpha =\beta$ yields
\[
\frac{\d}{\d \tau}  \calE_p( \bfu \, \vert \, \bfU)
= - \calD_{p,\rm{diff}}(\rho,\zeta)  - \ee^\tau \Dpreact(\rho,\zeta) ,
\]
where the reactive dissipation takes the form
\begin{align*}
\Dpreact (\rho,\zeta) 
&=  \int_\R k U^\alpha \alpha \big( F_p'(\zeta) 
	-  F_p'(\rho) \big) (\zeta^\alpha{-}\rho^\alpha) \dd y  \\
&= \int_\R k  U^\alpha  \frac{\alpha}{p{-}1}(\zeta^{p{-}1} {-} \rho^{p{-}1})
( \zeta^\alpha {-} \rho^\alpha) \dd y \geq 0,
\end{align*}
and where for the diffusive part a similar integration by parts gives
\begin{align*}
 \calD_{p,\rm{diff}} &(\rho,\zeta) =  \IpFisher (\rho,\zeta) 
 + \frac12 \calE_p( \bfu \, \vert \, \bfU) \\
 & \hspace*{1.3cm} - \int_\R   \big(F_p'(\rho) \rho - F_p(\rho) \big)
	\big( d_1 U'' + \frac{y}{2} U' \big) +  \big(F_p'(\zeta) \zeta - F_p(\zeta)\big)
	 \big( d_2 V'' + \frac{y}{2} V' \big) \dd y \\
&=  \IpFisher (\rho,\zeta) 
 + \frac12 \calE_p( \bfu \, \vert \, \bfU) 
 - \int_\R  \Big( \big(F_p'(\zeta) \zeta - F_p(\zeta)\big)
	 - \big(F_p'(\rho) \rho - F_p(\rho) \big) \Big) \alpha \Lambda \dd y,
\end{align*}
since $\alpha = \beta$. Further, with
$ F_p'(\rho) \rho - F_p(\rho) = \frac1p (\rho^p{-1})$, this leads to
\begin{align*}
 \calD_{p,\rm{diff}}(\rho,\zeta) &=  \IpFisher (\rho,\zeta) 
 + \frac12 \calE_p( \bfu \, \vert \, \bfU) 
  - \int_\R \frac1p (\zeta^p{-}\rho^p) \alpha \Lambda \dd y.
\end{align*}
\end{proof}

\subsection{Decay estimates}
\label{su:DecayEstimates}
 
After deriving the entropy-dissipation relation \eqref{eq:EntDissLaw}, the
\EEE next step requires to find a so-called entropy-dissipation estimate, that
is an estimate of the form $\calD_\phi (\bfu) \geq \Psi (\calE_\phi(\bfu))$ for
a non-negative function $\Psi$. Under appropriate assumptions on $\Psi$, this
usually gives exponential convergence to the equilibrium, where for specific
$\Psi$ the rate can be estimated explicitly.  For instance, if one even obtains
the inequality $\calD_\phi (\bfu) \geq \eta \, \calE_\phi(\bfu)$ for a positive
constant $\eta$, one can easily see that $\eta$ is exactly the desired rate by
using Gr\"onwall's inequality.  Since the dissipation functional $\calD_\phi$
from Proposition \ref{prop:Dissipation} has no fixed sign due to the mixed
term, these bounds above can scarcely be expected for the given problem. But in
fact, the nonnegativity for \textit{all} times is not a necessary assumption to
obtain convergence.  If a dissipation functional without sign can be estimated
by
\begin{equation}
\label{eq:EntDissIneq}
\calD_\phi(\bfu) \geq  \eta \,   \calE_\phi (\bfu) - K \ee^{-\gamma \tau},
\end{equation}
for example, with $\gamma >0$ and $K\geq 0$, then this will still yield
convergence with a rate that is the minimum of $\eta$ and $\gamma$ (see Lemma
\ref{le:Gronwall} below for the precise statement). The non-negative function
$\tau \mapsto K \ee^{-\gamma \tau}$ can be interpreted then as an upper bound
for the relative entropy for not being a true Lyapunov function.  Since the
function decays exponentially in time, this error is well-behaved when time is
large enough.  We will see later that in some cases, namely for the
stoichiometric coefficients fulfilling $\alpha,\beta \geq 2$, the inequality
\eqref{eq:EntDissIneq} is exactly what we will prove for the given dissipation
$\calD_\rmB$ from Proposition \ref{prop:Dissipation}. In the other case for
$\alpha,\beta \in [1,2)$, we need the following more general statement.

%\begin{lemma} \label{le:Gronwall}
%Let $E:[0,\infty) \to \R \cup \{ \infty \}$ be a function in time satisfying $E(0) < \infty$ and
%the ordinary differential inequality 
%\[
%\frac{\d}{\d \tau}E (\tau) \leq - (\eta {-} \mu \ee^{-\tau}) \,  E (\tau) 
%	+ K \ee^{-\gamma \tau}
%\]
%for $\eta,\gamma >0$, $\mu \geq 0$ and $K \in \R$. Then we have $E(\tau^*) < \infty$ 
%for all $\tau^* >0$ and
%%\[
%%E (\tau) \leq  \ee^{- \min \{ \eta,\gamma\} \tau + \mu} \big( E (\tau^*)
%%	+ C(\tau) \big) \quad \text{ for all } \tau \geq \tau^*,
%%\]
%\[
%E (\tau) \leq \ee^{\mu(\ee^{-\tau^*}{-} \ee^{-\tau})} \, \ee^{-\eta(\tau{-}\tau^*)}
%\Big( E (\tau^*)+ K \ee^{-\eta \tau^*} \int_{\tau^*}^\tau \ee^{-(\gamma{-}\eta)s} \dd s
%\Big) \quad \text{ for all } \tau \geq \tau^*.
%\]
%\end{lemma}

\begin{lemma} \label{le:Gronwall}
Let $E:[0,\infty) \to \R \cup \{ \infty \}$ be a function satisfying $E(0) < \infty$ and
the ordinary differential inequality 
\[
\frac{\d}{\d \tau}E (\tau) \leq - (\eta {-} \mu \ee^{-\tau}) \,  E (\tau) 
	+ K \ee^{-\gamma \tau}
\]
for $\eta,\gamma >0$ and $K,\mu \geq 0$. Then we have $E(\tau) < \infty$ 
for all $\tau >0$ and
%\[
%E (\tau) \leq \ee^\mu \Big( E(0)\, \ee^{- \eta \tau} 
%+ K \, \ee^{ -\eta \tau} \int_{0}^\tau \, \ee^{(\eta {-} \gamma)s } \dd s  \Big)
%\leq \ee^{- \min \{\eta,\gamma \}\tau +  \mu} \Big( E(0)
%+ C(\tau )\Big), 
%\]
\[
E (\tau) \leq \ee^{-\eta \tau +\mu} \big( E(0)+ R(\tau) \big) \quad \text{ with } \quad
R(\tau) :=  K  \int_{0}^\tau \, \ee^{(\eta {-} \gamma)s } \dd s .
\]
%where $C(\tau) := 2K \vert \eta {-} \gamma \vert^{-1}$ if $\eta \neq \gamma$ and
%$C(\tau) := K\tau $ for $\eta = \gamma$.
Calculating the function $R$ gives 
$E(\tau) \leq \ee^{- \min \{\eta,\gamma \}\tau +  \mu} 
\big( E(0)+ 2K \vert \eta - \gamma \vert^{-1} \big)$ if $\eta \neq \gamma$ and
$E(\tau) \leq \ee^{- \eta \tau + \mu} \big( E(0)+ K \tau\big)$ in the case $\eta = \gamma$.
\end{lemma}
The proof of this lemma can be found in the appendix.  We see that this weaker
version of an entropy-dissipation estimate is enough to obtain the desired
convergence \eqref{eq:EntropyConvergence}. In the following sections this
differential inequality is exactly what we want to derive for the relative
entropy as a function of time.

\section{Convergence for the special case $\alpha=\beta$}
\label{se:SpecialCase}

Let us begin with considering the special case $\alpha=\beta \geq 1$. That is, we study the 
solutions $\bfu = (u,v)^\top$ of the reaction-diffusion system
\begin{equation} \label{eq:RDS.specialCase}
\begin{aligned}
u_\tau &= d_1 u_{yy} + \frac{y}{2} u_y+ \ee^\tau \alpha k (v^\alpha {-} u^\alpha), \\
v_\tau &= d_2 v_{yy} + \frac{y}{2}v_y - \ee^\tau \alpha k (v^\alpha {-} u^\alpha)
\end{aligned}
\end{equation}
together with continuous initial data $\bfu^0 =(u(0),v(0))^\top$ fulfilling the asymptotic 
boundary conditions
\begin{equation} \label{eq:BC.specialCase}
\bfu^0(\pm \infty) = (A_\pm^\alpha, A_\pm^\alpha)^\top.
\end{equation}
In this special case,  the mixed term $\Imix $ from \eqref{eq:Imix.general}
simplifies significantly, namely
\[
\Imix(\rho,\zeta) = \int_\R \big(\zeta-\rho\big) \,\alpha \,\Lambda \dd y. 
\]
The main point is that $\Dreact$ is able to control $\zeta{-}\rho$ through the
term $\Gamma(\rho^\alpha,\zeta^\alpha)$, which will be part of 
Section \ref{su:al=be.Boltz}, where we focus on the Boltzmann entropy.
Afterwards, in Section \ref{su:al=be.GenEntropy}, we will allow more general 
entropy functions and aim to control $\Ipmix$ with $\Dpreact$
from Proposition \ref{prop:Diss.gen.Entr} in a similar way.

Another much less important point is that the profile equation
\eqref{eq:SelfSimProfile} simplifies also significantly, such that $U,\ V$, and
$\Lambda$ can be solved explicitly. Indeed,  by inserting $\alpha=\beta$
one can see that the profile $\bfU =(U,V)^\top$ is characterized by 
solving the linear ODE 
\begin{equation} 
  \label{eq:Profile.specialCase}
\begin{aligned}
& \frac{d_1{+}d_2}{2} \; U''(y) + \frac{y}{2} U'(y)  = 0 \quad \text{with } \ 
 U(\pm \infty) = A_\pm^\alpha ,\\
& V(y)=U(y)\quad \text{and} \quad  \Lambda = \frac{d_2{-}d_1 } {2\alpha} \: U''(y). 
\end{aligned}
\end{equation}
This means that $U$ is of error-function type like the profile
\eqref{eq:LinDiff.profile} for the linear diffusion equation, but with respect
to the average of the diffusivities $(d_1{+}d_2)/2$. However, we do not need
this outcome in the following calculations; thus, we do not close the door for
further generalizations as this is not true if $\alpha \neq\beta$.  The aim of
this section is twofold: in Section \ref{su:al=be.Boltz} we show
exponential convergence for the case of the relative Boltzmann entropy $\EB$,
and in Section \ref{su:al=be.GenEntropy} we show that in this case estimates
with different relative entropies are possible and even more advantageous.   
In both cases, the result is obtained by proving a suitable bound for the
dissipation functional, like inequality \eqref{eq:EntDissIneq}.

\subsection{The case $\alpha=\beta\geq 1$ with Boltzmann entropy}
\label{su:al=be.Boltz}

Here we restrict to the case that $\phi$ is given by the Boltzmann
function $\LB(z) = z \log z - z +1$, which is intrinsically linked to reaction
diffusion systems, see e.g.\ \cite{DesFel06EDTE, Miel11GSRD} and the recent
justification via Large Deviation principles in \, \cite{MiPeRe14RGFL,
  MPPR17NETP, Mitt18EMQC}.

Our convergence result reads as follows.

\begin{theorem}[Convergence for $\alpha=\beta \geq 1$ with Boltzmann entropy]
\label{thm:specialCase}
Consider the relative Boltzmann entropy
$E(\tau) := \EB ( \bfu (\tau) \, \vert \, \bfU)$ for the unique similarity
profile $\bfU$ that solves \eqref{eq:Profile.specialCase}.
Then, for all solutions $\bfu$ of the scaled system \eqref{eq:RDS.specialCase}
with $E(0) < \infty$, the following differential inequalities are satisfied:
\begin{subequations}
\label{eq:al=be.Estim}
\begin{align}
\label{eq:al=be=1.Estim}
&\text{For $\alpha =1$, it holds} \quad
 &&\dot E(\tau) 
\leq - \big(\frac12  - \mu_0 \ee^{-\tau} \big) E(\tau) + K_0 \, \ee^{-\tau} 
\quad \text{ for all } \tau>0, \\
\label{eq:al=be1-2.Estim}
&\text{for $1 < \alpha < 2$, we have}  \quad 
&& \dot E(\tau) 
\leq - \big(\frac12  - \mu_1 \ee^{-\tau} \big) E(\tau) + K_1 \, \ee^{-\tau} 
\quad \text{ for all } \tau>0,\\
\label{eq:al=be.ge2.Estim}
&\text{and if $\alpha \geq 2$, then } \quad 
&& \dot E(\tau) 
\leq - \frac12  E(\tau) + K_2 \, \ee^{-\tau/(\alpha {-} 1)} \quad \text{ for all } \tau>0,
\end{align}
\end{subequations}
where all constants $\mu_0,\mu_1,K_0,K_1$, and $K_2$ only depend on the problem
data and are precisely defined in Lemmas \ref{le:Control.Imix.ngeq2},
\ref{le:Control.Imix.1n2}, and \ref{le:Control.Imix.n1}, respectively.
\end{theorem}

\AAA Here we provide estimates for the relative Boltzmann entropy $\EB=\calE_1$
only and refer to Theorem \ref{thm:Conv.E12} and Corollary
\ref{co:Decay.Ep.al=be} for relative entropies $\calE_p(\bfu|\bfU)$. \EEE  

Notice that all constants above are explicit and depend only on the given data
and not on the solutions. The proof of this result relies on a series of
lemmas and will be completed at the end of this section.  We will see that
the essential step in the case $\alpha=\beta$ is to use that $\Imix $ can
be written as a function of $\rho-\zeta$ and hence can be controlled by
$\Dreact$ alone. This simplifies the analysis and gives better convergence
results. In particular, we do not need additional assumptions on the
similarity profile $\bfU$, as will be needed in Section
\ref{se:GeneralCase}.  We start by summarizing our results on the mixed term
discussed above.

\begin{lemma}
\label{le:SpecialCase.Imix}
In the case $\alpha=\beta$, the mixed term from Proposition
\ref{prop:Dissipation} reduces to
\begin{equation*}
\Imix(\rho,\zeta) =  \int_\R \big( \zeta(y){-} \rho(y)\big) \,\alpha\, \Lambda(y) \dd y,
\quad \text{ where } \quad \Lambda(y) = \frac{d_2 {-} d_1}{2\,\alpha}\: U''(y).
\end{equation*}
\end{lemma}
Notice that the simplified mixed term $\Imix$ vanishes if we have equal
diffusivities $d_1 = d_2$ as then $\Lambda \equiv 0$.  This means that in the
very special case where additionally to $\alpha=\beta$ also the diffusivities
coincide, $\EB$ is a true Lyapunov function, and we
have an explicit decay rate through the bonus factor $1/2$. 

\begin{corollary}
\label{cor:EqualDiff}
In addition to the assumptions of Theorem \ref{thm:specialCase}, assume
$d_1 = d_2$ and $\alpha=\beta\geq 1$. Then, we obtain exponential convergence
of all solutions $\bfu$ to the profile $\bfU$:
\begin{align*}
  \EB ( \bfu (\tau )\, \vert \,  \bfU) \leq \ee^{-\tau/2 } \,\EB ( \bfu (0) \,
  \vert \,  \bfU)   \text{ for } \tau >0.  \bigskip
\end{align*}
\end{corollary}

Let us continue with two diffusivities $d_1,d_2 >0$ that do not coincide in
general.  At first, the dissipation functional can naively be estimated by
omitting the Fisher information
\begin{align*}
\calD_\rmB(\rho,\zeta) 
& \geq 	\frac12  \EB ( \bfu \, \vert \,  \bfU) 
	-  \Imix(\rho,\zeta) + \ee^\tau \,\Dreact(\rho,\zeta).
\end{align*}
Since we have the bonus factor, we are not dependent on exploiting the Fisher
information in order to obtain a qualitative convergence result. Most often,
estimation of the Fisher information, for example by using the Logarithmic
Sobolev inequalities, leads to the fact that the dissipation functional can be
bounded in terms of the relative entropy. Thanks to the parabolic scaling, 
the corresponding term is $\frac12\EB(\bfu|\bfU)$, so we can drop the Fisher
information, in contrast to \cite{Grog83ABSC, DesFel06EDTE, Miel17UEDR,
  MieMit18CEER}, where the unscaled system \eqref{eq:RDS} is studied on
bounded domains. Even more, the fact that we can consider unbounded domains at
all is precisely due to the scaling and the resulting bonus factor.
Nevertheless, it might be possible to improve the estimates if the Fisher
information can be used. But in contrast to bounded domains, it seems to be
much more complicated on the unbounded domain $\R$. And to the authors' best
knowledge no way is found until now.

The idea is now to control the mixed term with the reactive dissipation and its
useful prefactor $\ee^\tau$. As mentioned earlier, the simple structure of
the mixed term in Proposition \ref{le:SpecialCase.Imix} makes it easier to
bound the dissipation functional from below. Indeed, Lemma
\ref{le:SpecialCase.Imix} implies 
\begin{align*}
\Imix(\rho,\zeta)  - \ee^\tau \Dreact(\rho,\zeta) 
&=  \int_\R ( \zeta{-} \rho) \alpha \Lambda \dd y  
	- \ee^\tau \int_\R k U^\alpha \Gamma(\rho^\alpha,\zeta^\alpha) \dd y \\
&=  \int_\R \rho 
\Big(\frac{\zeta}{\rho} {-} 1\Big) \alpha\Lambda 
	-  \ee^\tau k (\rho U)^\alpha \Big( \frac{\zeta^\alpha}{\rho^\alpha}{-} 1 \Big)
	\log(\zeta^\alpha / \rho^\alpha) \dd y.
\end{align*}
Next, we set $z:= \frac{\zeta}{\rho} -1$ as a new auxiliary variable and
define, for $\alpha \geq 1$, the following family of functions
\begin{equation} \label{eq:Phi.alpha}
\Phi_\alpha(z):= \begin{cases}  
  \big( (z{+}1)^\alpha{-}1 \big)\log \big((z{+}1)^\alpha \big) 
	&\text{for } z > -1,\\
	+\infty &\text{for } z \leq -1.
\end{cases}
\end{equation}

Note that for all $\tau >0$ and $y \in \R$, we have $z >-1$. We only need to
extend $\Phi_\alpha$ for technical reasons.  This leads to
\begin{align*}
\Imix(\rho,\zeta)  - \ee^\tau \Dreact(\rho,\zeta) 
&= \int_\R \alpha\Lambda \rho z 
	-  \ee^\tau k (\rho U)^\alpha  \Phi_\alpha(z) \dd y.
\end{align*}
With respect to the auxiliary variable $z$, the integrand can be seen as the
difference of a linear term and the function $\Phi_\alpha$. Remember that for a
(not necessarily convex) function $\Phi$, its Legendre transform $\Phi^*$ is
defined as
$\Phi^*(\xi) := \sup_z \big\{ \langle \xi, z \rangle - \Phi(z) \big\}$.  Thus,
we obtain
\begin{align} \label{eq:Estimate.Imix.Phi.n}
&\Imix(\rho,\zeta)  - \ee^\tau \Dreact(\rho,\zeta) 
 \leq \int_\R   \ee^\tau\, k\, (\rho U)^\alpha \:  \Phi_\alpha^*\Big( 
	\frac{\alpha\Lambda}{k U^\alpha} \rho^{1{-}\alpha} \ee^{-\tau}\Big) \dd y \notag \\
&\quad = \int_\R   \ee^\tau \,k \,(\rho U)^\alpha\, \Phi_\alpha^*\Big(\tilde \Lambda 
	 \rho^{1{-}\alpha} \ee^{-\tau} \Big) \dd y \quad \text{with } 
\tilde \Lambda (y) := \frac{\alpha\Lambda(y)}{k U^\alpha(y)} 
= \frac{(d_2{-}d_1) }{2k} \frac{U''(y)}{U(y)^\alpha} .
\end{align}
Unfortunately, the Legendre transform $\Phi_\alpha^*$ cannot be calculated
explicitly, but a suitable estimate of $\Phi_\alpha^*$ from
above is sufficient to continue with \eqref{eq:Estimate.Imix.Phi.n}.  For this,
we have the following auxiliary result, which is proved in Appendix
\ref{se:Appendix}.

\begin{lemma}
\label{le:phi.n.Legendre}
Consider for $\alpha \geq 1$ the function $\Phi_\alpha$ defined in
\eqref{eq:Phi.alpha}.  Its Legendre transform $\Phi_\alpha^*$ satisfies, for
all $\xi\in\R$, the following estimates, where
$ \wt c_\alpha=\big(\frac2{\alpha^2}\big)^{1/(\alpha-1)}
\frac{\alpha{-}1}\alpha$:    
\begin{enumerate}
\item For \ $\alpha=1$ \quad we have \quad  $\Phi_1^*(\xi) \leq \ee^{\xi} -
  \xi -1$;  
\item for $\alpha \in (1,2]$ \quad it holds \quad 
  $\Phi_\alpha^*(\xi) \leq  \max \big\{\wt c_\alpha\,|\xi|^{\alpha/(\alpha-1)}
  ,\, \frac1{2\alpha} \, \xi^2 \big\} $; 
\item and if \ $\alpha \geq 2$ \quad then \quad 
$\Phi_\alpha^*(\xi) \leq \wt c_\alpha \, \vert \xi \vert^{\alpha/(\alpha{-}1)} $.
\end{enumerate}
Moreover, for all $\alpha\geq 1$ we have $\Phi_\alpha^*(\xi)\leq
\frac1{2\alpha}\,\xi^2$ for $|\xi|\leq \alpha$. 
\end{lemma}

% The proof of this lemma can be found in the appendix.  We remark that these
% estimates are not optimal, but they give constants which are practical and
% sufficient for our analysis. Looking into the proof, one might see that they
% can be improved by finding better bounds for the function $\Phi_\alpha$ from
% below, but these give only slightly better estimates in the end.

The bounds on the Legendre transform $\Phi_\alpha^*$
will help us to find a bound for the dissipation functional.  We start with the
mathematically easier case $\alpha=\beta \geq 2$.  In this case the dissipation
functional fulfills the estimate \eqref{eq:EntDissIneq}, which is the inequality
from Lemma \ref{le:Gronwall} with $\mu=0$.  The other case 
$\alpha \in [1,2)$ will be treated afterwards. \EEE

\begin{lemma}
\label{le:Control.Imix.ngeq2}
Let $\alpha=\beta \geq 2$. The dissipation functional $\calD_\rmB$ from Proposition 
\ref{prop:Dissipation} can be bounded from below by
\begin{equation*}
\Imix(\rho,\zeta) -\ee^\tau\Dreact(\rho,\zeta) \leq  K_2 \, \ee^{-\tau/(\alpha{-}1)}  \  \text{ for } 
K_2 :=  \frac{\wt c_\alpha}{ k^{1/(\alpha{-}1 ) } }  \int_\R 
\Big\vert \frac{\alpha\Lambda(y)}{U(y)}\Big\vert^{\alpha/(\alpha{-}1)}  \dd y < \infty
\end{equation*}
with $\wt c_\alpha$ from Lemma  \ref{le:phi.n.Legendre}. 
\end{lemma}
\begin{proof}
  We start from estimate \eqref{eq:Estimate.Imix.Phi.n} and insert the upper
  estimate for $\Phi_\alpha^*$  from Lemma
  \ref{le:phi.n.Legendre} (case $\alpha \geq 2$) to arrive at 
\begin{align*}
 \Imix (\rho,\zeta)  - \ee^\tau \Dreact (\rho,\zeta)  
 & \leq  \ee^\tau k \int_\Omega (\rho U)^\alpha \:\wt c_\alpha
 \Big| \tilde \Lambda \,\rho^{1-\alpha} \ee^{-\tau} \Big|^{\alpha/(\alpha-1)} \dd y 
\\
&=   \ee^{-\tau/(\alpha-1)}  \frac{\wt c_\alpha}{ k^{1/(\alpha{-}1 ) } }  \int
     \Big|\frac{\alpha\Lambda(y)}{U(y)} \Big|^{\alpha/(\alpha-1)} \dd y\  
 = \  K_2\, \ee^{-\tau/(\alpha{-}1) }.
\end{align*}
The dependence on $\rho$ is exactly canceled out, such that the
assertion is established. 
\end{proof}

Although we certainly lose some optimality in estimating the function
$\Phi_\alpha^*$, we see that we get a profitably bound.  Estimating by a
function with exponent $\alpha/(\alpha{-}1)$ is the only choice that leads to a
uniform bound for all solutions, because only then $\rho$ cancels out. 
However, we obtain a decay rate $\ee^{-\tau/(\alpha-1)}$ only, which is
not really optimal in terms of decay for $\tau\to \infty$ as is shown in the
following remark. But it has the advantage that it is valid globally, i.e.\ for
all solutions.

\begin{remark}[Improved decay rate]
\label{re:ImprovDecayRate} 
Using the exponential convergence of $u=\rho U$ to $U$ (with the smaller
\EEE decay rate from above) and parabolic regularity theory (involving the term
$\IFisher$ dropped so far), it is possible to show that
$\rho(\tau,y)\in [\underline{c}, \ol c] $ for all $y\in \R$ and
$\tau\geq \tau_1$, where $0 < \underline{c}<1< \ol c < \infty$ and $\tau_1$ may
depend on $\rho$. Thus, we can use the better quadratic estimate
$\frac1{2\alpha} \, \xi^2$ for $\Phi_\alpha^*(\xi)$ for $|\xi|\leq \alpha$, see
the end of Lemma \ref{le:phi.n.Legendre}.  Setting
$\tau_2=\log\big(\|\tilde \Lambda\|_\infty /(\alpha \underline{c}^{\alpha-1})
\big) $ we obtain for $\tau \geq \max\{\tau_1,\tau_2\}$ the better decay
estimate
\begin{align*}
 \Imix (\rho,\zeta)  - \ee^\tau \Dreact (\rho,\zeta)  
 & \leq  \ee^\tau k \int_\Omega  \frac{(\rho U)^\alpha} {2\alpha} 
 \Big| \frac{\tilde \Lambda \, \ee^{-\tau}} {\rho^{\alpha-1}} \Big|^2 \dd y 
%\\ & 
\leq \ee^{-\tau}\,\frac\alpha{2\,k}\:\frac{\ol
     c^\alpha}{\underline{c}^{\alpha_1}} \int \Lambda(y)^2
     U(y)^{\alpha-2} \dd y .
\end{align*}
\AAA Another way of deriving the optimal decay like $\ee^{-\tau/2}$ is
given in Corollary \ref{co:Decay.Ep.al=be}, where $\EB=\calE_1$ is replaced by
the higher order entropies $\calE_p$ with $p=\alpha{-}1$, see estimate
\eqref{eq:Ep.al=be.ge2}.  
\end{remark}
 
With Lemma \ref{le:Gronwall}, we identified a bound for the dissipation
functional that still yields the desired convergence although its sign is not
necessarily non-negative for all times.  In the previous proof, we
obtained the estimate \eqref{eq:al=be.Estim} with $\mu_j = 0$. Next, we study
the cases $\alpha =\beta \in (1,2)$ and $\alpha=\beta=1$ and will see that
the additional term $\mu_j\ee^{-\tau}$ will appear then. That is because
there will be some terms containing $\rho$ that cannot be estimated uniformly,
so they need to be estimated by the relative entropy.

\begin{lemma} 
\label{le:Control.Imix.1n2}
Let $\alpha=\beta \in (1,2)$. Then for all times $\tau >0$ the dissipation
functional $\calD_\rmB$ can be bounded from below by
\begin{equation*}
\Imix(\rho,\zeta) -\ee^\tau\Dreact(\rho,\zeta) \leq  \mu_1 \; \ee^{-\tau}
\EB ( \bfu \, \vert \,  \bfU) + K_1\, \ee^{-\tau},
\end{equation*}
where the constants $\mu_1$ and $K_1$ are given by 
\[
\mu_1 = \frac1k\,\Big\Vert \frac{\alpha^2\Lambda^2}{U^{3-\alpha}} \Big\|_{\rmL^\infty}
\quad \text{and} \quad 
K_1= \int_\R\Big( \frac{\alpha^2\Lambda^2}{k U^{2-\alpha}} + 
\frac{\wt c_\alpha} {k^{1/(\alpha-1) } } 
  \Big| \frac{\alpha\Lambda} U\Big|^{\alpha/(\alpha-1)}  \Big) \dd y   .
\]
\end{lemma}
\begin{proof}
We again start with the estimate \eqref{eq:Estimate.Imix.Phi.n} and insert the
upper estimate for $\Phi_\alpha^*$ as derived in Lemma \ref{le:phi.n.Legendre},
where we estimate $\max\{a,b\}\leq a+ b$:
\begin{align*}
\Imix &(\rho,\zeta)  - \ee^\tau \Dreact(\rho,\zeta)  
\leq \int_\R   \ee^\tau k U^\alpha \rho^\alpha \Phi_\alpha^*\Big( 
  \tilde \Lambda \, \rho^{1{-}\alpha}\,\ee^{-\tau} \Big) \dd y
\\
&\leq  \int_\R   \ee^\tau\, k\,(\rho U)^\alpha \Big(  \frac1{2\alpha}
\tilde\Lambda^2 \rho^{2-2\alpha} \ee^{-\tau} + \wt c_\alpha
\big|\tilde\Lambda\big|^{\alpha/(\alpha-1)} \rho^{-\alpha} \,\ee^{-\tau
  \alpha/(\alpha-1)} \Big) \dd y 
\\
& \leq  \ee^{-\tau} \: \int_\R \frac{\alpha^2\Lambda^2}{k\,U^{2-\alpha}}
\,\frac{\rho^{2-\alpha}} {2\alpha}  \dd y \ + \ \ee^{-\tau/(\alpha-1)} 
   \int_\R \frac{\wt c_\alpha} {k^{1/(\alpha-1) } } 
  \Big| \frac{\alpha\Lambda} U\Big|^{\alpha/(\alpha-1)} \dd y .
\end{align*}
In the second term we can estimate $ \ee^{-\tau/(\alpha-1)} \leq 
\ee^{-\tau}$ because of $\alpha\in (1,2)$.
In the first term we still need to estimate $\rho^{2-\alpha}$ where the
exponent is less than 1. For this we use $ \rho^{2-\alpha}/(2\alpha) \leq
\LB(\rho)+1$ for $\rho\geq 0$ and $\alpha \in [1,2]$ and obtain
\begin{align*}
\Imix& (\rho,\zeta)  - \ee^\tau \Dreact(\rho,\zeta)  
 \leq   \ee^{-\tau}   \int_\R \Big( \frac{\alpha^2\Lambda^2}{k U^{2-\alpha}} \big(
 \LB(\rho){+}1 \big) +   \frac{\wt c_\alpha} {k^{1/(\alpha-1) } } 
 \Big| \frac{\alpha\Lambda} U\Big|^{\alpha/(\alpha-1)}\Big)\dd y  
\\
& \leq \ee^{-\tau} \,\frac1k\, \Vert \alpha^2\Lambda^2/U^{3-\alpha}   \Vert_\infty
 \int_\R  \LB(\rho) \, U  \dd y + 
\ee^{-\tau} \, K_1,
\end{align*}
with $K_1$ as in the assertion. The desired result follows from $\int_\R
\LB(\rho) \, U  \dd y \leq \EB(\bfu \, | \, \bfU)$. 
\end{proof}

The remaining case $\alpha=\beta=1$ is important as this linear case
relates to the case of Markov semigroups. We proceed similarly as above but obtain a
rather large bound because $\Phi_1^*$ has exponential growth. A better bound for this
case is obtained in Section \ref{su:al=be.GenEntropy}.  
 
\begin{lemma}
\label{le:Control.Imix.n1}
For  $\alpha=\beta=1$ the dissipation functional $\calD_\rmB$
 can be bounded from below by
\begin{equation*}
 \Imix(\rho,\zeta) -\ee^\tau\Dreact(\rho,\zeta) \leq  
 \mu_0 \; \ee^{-\tau} \EB ( \bfu \, \vert \,  \bfU)  + K_0 \, \ee^{-\tau},
\end{equation*}
where the constants $\mu_0$ and $K_0$ are given in terms of
$\lambda^*:=\big\| \Lambda/U\big\|_{\rmL^\infty}$  by 
\[
\mu_0 = \frac{(\lambda^*)^2\,\ee^{\lambda^*/k}}{2k} \quad \text{and} \quad 
K_0 = \frac{\ee^{\lambda^*/k}}{k} \int_\R \frac{\alpha^2\Lambda^2}{U} \dd y. 
\]
\end{lemma}
\begin{proof}
As before we start from \eqref{eq:Estimate.Imix.Phi.n} and now need to estimate
$\Phi_1^*(\tilde\Lambda \, \rho^0\, \ee^{-\tau})$. The decisive advantage is that
$\rho^0 = 1$ provides automatically a bound independently of $\rho$, and the
exponential growth does not harm too much. 

Clearly, we have $|\tilde\Lambda(y)\ee^{-\tau}|\leq \lambda^*/k$ for all
$\tau\geq 0$ and $y\in \R$. Using $(\Phi_1^*)''(\xi)=\ee^{\xi}$ we obtain the
quadratic upper estimate
\[
 \Phi_1^*(\xi) \leq \frac{\ee^{\lambda^*/k}}2 \:\xi^2 \quad 
\text{for } |\xi|\leq \lambda^*/k.
\]
Inserting this into \eqref{eq:Estimate.Imix.Phi.n} first and using $\rho \leq
\LB(\rho){+}2$  we find
\begin{align*}
\Imix(\rho,\zeta)  - &\ee^\tau \Dreact(\rho,\zeta) 
 \leq  \int_\R \ee^\tau k (\rho U) \;\frac{\ee^{\lambda^*/k}}2\, \tilde
 \Lambda^2 \,\ee^{-2\tau}  \dd y \\
& \leq \ee^{-\tau} \:\frac{\ee^{\lambda^*/k}}{2k}\int_\R   \big(\LB(\rho){+} 2\big)
U \,\frac{\alpha^2\Lambda^2}{U^2} \dd y  \ \leq \ \ee^{-\tau}\,\mu_0\, \EB(\bfu\, | \,
\, \bfU) + K_0\ee^{-\tau}, 
\end{align*}
which is the desired result. \medskip
\end{proof}

\AAA The estimate in Lemma \ref{le:Control.Imix.n1} has rather large constants
because of the term $\ee^{\lambda^*/k}$. Since the linear case has many
applications, in particular as Kolmogorov forward equation for Markov processes,
we provide a better bound in Corollary \ref{co:Decay.Ep.al=be}. There we
replace the Boltzmann entropy $\EB=\calE_1$ by the relative entropy
$\calE_p$ with $p=1/2$, which gives exactly the Hellinger distance, see
\eqref{eq:Helli.calEp}. 
\EEE

As we have covered now all the cases $\alpha=\beta\geq 1$ we are now
ready to summarize which completes the proof of our main result.
\\[0.4em]
\begin{proof}[Proof of Theorem \ref{thm:specialCase}]
Abbreviate the relative Boltzmann entropy by 
$E(\tau):=  \EB ( \bfu (\tau) \, \vert \,  \bfU)$.
At first, Proposition \ref{prop:Dissipation} and the nonnegativity of the
Fisher information give
\[
 \frac{\d}{\d \tau}E = - \calD_\rmB 
 	\quad \text{with }\calD_\rmB \geq \frac12 E + \ee^\tau \, \Dreact   - \Imix.
\]
Since $\alpha {=} \beta$,  the mixed term reduces to
$\Imix (\rho, \zeta) = \int_\R \big( \zeta {-} \rho \big) \alpha \Lambda \dd y$.
In all three cases for $\alpha$ we have shown 
$\Imix- \ee^\tau \Dreact \leq \mu_j\ee^{-\tau}
E(\tau) + K_j \ee^{-\sigma_j \tau}$ with $\mu_2=0$. Inserting this we arrive exactly at
\eqref{eq:al=be.Estim}, and our result is established.
\end{proof}

\subsection{The case $\alpha = \beta$ with general entropies}
\label{su:al=be.GenEntropy}

While for the case $\alpha \neq \beta$ it is really necessary to take the
Boltzmann entropy, we have more flexibility with the choice of entropy
functions if $ \alpha = \beta$.  In this section, we choose general entropies
which will improve the results, in particular for the case $\alpha=\beta=1$.

Recall the family of entropy functions $F_p$ introduced in \eqref{eq:F.p} 
and consider $p \not \in \{0,1\}$ so that all the formulas are well-defined.
Of course, it is possible to consider the cases $p =0$ and $p =1$ by passing 
to the limit and use that $\lim_{p \to 0} \frac 1p (\zeta^p {-}1) = \log(\zeta)$.

We can define the relative entropy associated to the function $F_p$ by
\[
\calE_p (\bfu (\tau)\, \vert \, \bfU) := \int_\R U(y) F_p(\rho(y)) + V(y) F_p(\zeta(y)) \dd y 
\quad \text{ with } \rho := u/U \text{ and } \zeta := v/V.
\]
From Proposition \ref{prop:Diss.gen.Entr} we know
\[
 \frac{\d}{\d \tau} \calE_p (\bfu \, \vert \, \bfU) = - \calD_p (\rho,\zeta)
 \leq - \frac12 \calE_p (\bfu \, \vert \, \bfU) 
 	- \ee^\tau \Dpreact (\rho,\zeta) + \Ipmix(\rho,\zeta).
\]
Thus, the strategy is to estimate the difference $\Ipmix - \ee^\tau \Dpreact$, like in 
the previous section.
We have 
\begin{align*}
\Ipmix(\rho,\zeta) &- \ee^\tau \Dpreact(\rho,\zeta) = 
\int_\R \frac1p \big(  \zeta^p {-} \rho^p \big)\, \alpha \Lambda
- \ee^\tau  k  U^\alpha  \frac{\alpha}{p{-}1}(\zeta^{p{-}1} {-} \rho^{p{-}1})
( \zeta^\alpha {-} \rho^\alpha) \dd y \\
&= \int_\R \frac1p \rho^p \Big(  \frac{\zeta^p}{\rho^p} {-} 1 \Big)\, \alpha \Lambda
- \ee^\tau  k  U^\alpha  \frac{\alpha}{p{-}1}
	 \rho^{\alpha{+}p{-}1} \Big(\frac{\zeta^{p{-}1}}{\rho^{p{-}1}} {-} 1 \Big)
\Big( \frac{\zeta^\alpha}{\rho^\alpha} {-} 1\Big) \dd y.
\end{align*}
Defining $z:=  \frac{\zeta^p}{\rho^p} - 1 > -1$ as an auxiliary variable and
the following two-parameter family of functions
\begin{equation} \label{eq:Phi.p.alpha}
\Phi_{p,\alpha}(z):= \begin{cases}  
 \frac{\alpha}{p{-}1} \; \big((z{+}1)^{\frac{p{-}1}{p}} {-}1 \big) 
 \big((z{+}1)^{\frac \alpha p} {-}1 \big)
	&\text{for } z > -1,\\
	+\infty &\text{for } z \leq -1.
\end{cases}
\end{equation}
as a generalization of \eqref{eq:Phi.alpha}, this yields
\begin{align}
  \label{eq:Est.Imix.Phi.p.alpha}
& \Ipmix  (\rho,\zeta) - \ee^\tau  \Dpreact(\rho,\zeta) 
\leq \int_\R \frac1p \rho^p \, \alpha \Lambda z
- \ee^\tau  k  U^\alpha  \rho^{\alpha{+}p{-}1} \,\Phi_{p,\alpha}(z) \dd y \notag \\
&\leq  \int_\R \ee^\tau k U^\alpha  \rho^{\alpha{+}p{-}1} \,
 \Phi_{p,\alpha}^*\big(\tfrac1p  \tilde \Lambda \rho^{1{-}\alpha} \ee^{-\tau} \big) \dd y
 \; \text{ with } \tilde \Lambda(y):= \frac{\alpha \Lambda(y)}{k U^\alpha(y)}
 = \frac{(d_2{-}d_1)}{2k} \frac{U''(y)}{U(y)^\alpha} .
\end{align}

From our theory above, we know that it is advantageous to estimate
$\Phi_{p,\alpha}$ from below by a \AAA multiple of $ z^2$, \EEE because then
$\Phi^*_{p,\alpha}$ has a quadratic upper bound. Hence, we prepare the following result.

\begin{lemma}[Quadratic bound for $\Phi^*_{p,\alpha}$]
\label{le:QuadraticBd}
For $\alpha>0$ and $p \in \big( 0,\max\{\alpha/2,\alpha{-}1\}\big]$ we have
\[
\wh M_{p,\alpha} := \sup\Bigset{ \frac{\alpha^2}{4p^2} \frac{z^2}{\Phi_{p,\alpha}(z)} }{ z\in
    \R\setminus\{0\} } <\infty
\]
and the quadratic upper bound $\Phi^*_{p,\alpha}(\zeta) \leq \AAA \wh M_{p,\alpha}
\big(\frac p\alpha \,\zeta\big)^2$ for all $\zeta\in \R$. 

\AAA In the given range we always have $\wh M_{p,\alpha}\geq 1/4$. 

For $p=1/2$ and $\alpha\geq 1$ we have
\AAA $\wh M_{1/2,\alpha} \leq \alpha/2$ and $\wh M_{1/2,1}=1/2$. 

For $p>0$ and $\alpha=p{+}1$ we have $\wh M_{p,p+1}= 1/4$. 
\end{lemma}
\begin{proof}
We fix a pair $(\alpha,p)$ in the given range and set
$f(z)=z^2/\Phi_{p,\alpha}(z)$ for $z\in (-1,0)\cup (0,\infty)$.

We first observe that $z\mapsto  \Phi_{p,\alpha}(z)$ behaves like $z^2$ for
$z\approx 0$. Hence, $f$ is bounded near $z=0$. Moreover, $ \Phi_{p,\alpha}$ is
bounded from below on intervals $(-1,-1{+}\delta)$ for small $\delta$. Since
$\Phi_{p,\alpha}$ is analytic in $(-1,0)$ and in $(0,\infty)$, the same is true
for $f$. Thus, $f$ is bounded on the interval $(-1,R)$ for all $R>0$. 
To obtain boundedness of $f$ it suffices to study its polynomial growth. Indeed,
we have $f(z)\sim z^\gamma$ with $\gamma= 2- \max\{0,(p{-}1)/p\} - \alpha/p$. 
However, the range for $(\alpha,p)$ was chosen exactly such that $\gamma\leq
0$, \AAA hence $\wh M_{p,\alpha}$ is finite. \EEE 

From $\Phi_{p,\alpha}(z)\geq \AAA \big(\frac\alpha p\, z\big)^2 / 
(4\wh M_{p,\alpha})$  we obtain \AAA
$\Phi_{p,\alpha}^*(\zeta) \leq \wh M_{p,\alpha} \big(\frac p\alpha
\,\zeta\big)^2$ by the properties of the Fenchel-Legendre transformation.
With $\Phi_{p,\alpha}(z)=\frac{\alpha^2}{p^2} z^2 + \mafo{h.o.t.}$, we
find $\wh M_{p,\alpha}\geq 1/4$. 

For the two explicit estimates we argue as follows. \EEE 
For $p=1/2$ and $\alpha\geq 1$ we have for all $z > -1$ the estimate \EEE
\begin{align}
\label{eq:Phi.1/2alpha}
 \Phi_{1/2, \alpha}(z)
 = 2\alpha \frac{z}{z{+}1} \big( (z{+}1)^{2\alpha} {-}1 \big)
 \geq  2\alpha  \frac{z}{z{+}1} \big( (z{+}1)^{2} {-}1 \big) 
 = \frac{2 \alpha z^2 (z{+}2)}{z{+}1} \geq 2 \alpha z^2,
\end{align}
\AAA where we use $(z{+}1)^{2\alpha}\geq (z{+}1)^2$ for $z>0$ and
$(z{+}1)^{2\alpha} \leq (z{+}1)^2$ for $z<0$. Thus, we have $\wh M_{1/2,\alpha}
\leq \alpha/2$. Using $\Phi_{1/2,\alpha}(z)=4\alpha^2 z^2 +\mafo{h.o.t.}$, we
obtain $\wh M_{1/2,\alpha}\geq 1/4$. For the case $\alpha=1$ we observe that
the first ``$\geq$'' in \eqref{eq:Phi.1/2alpha} is an equility, such that
$\Phi_{1/2,1}(z) \geq 2z^2$ is optimal, and $\wh M_{1/2,1}=1/2$ follows.

For the case $\alpha=p{+}1>1$ we set $\lambda = 1/p >0$ and
$w=z{+}1$. For $w>0$ we have
\begin{align*}
\Phi_{p,p+1}(z)& = \tfrac{p{+}1}{p{-}1}\big( w^{(p-1)/p}-1\big) 
\big( w^{(p+1)/p}-1\big)\\
& =\tfrac{(p{+}1)^2}{p^2} \,G_\lambda(w) \text{ with } 
G_\lambda(w)= \tfrac1{1{-}\lambda^2}\big( w^2-w^{1-\lambda} - w^{1+\lambda} + 1\big).
\end{align*}
Calculating the first three derivatives of $G_\lambda$, we find
$G'''_\lambda(w) = \lambda(w^{\lambda-2}- w^{-\lambda-2})$ and conclude
$G''_\lambda(w) \geq G''_\lambda(1)=2$. This implies $G_\lambda(w)\geq
(w{-}1)^2$ and hence $\Phi_{p,p+1}(z) \geq \frac{(p{+}1)^2}{p^2}\,z^2$. This
provides $\wh M_{p,p+1}= 1/4$, because optimality follows by taking $z\to 0$. 
\end{proof}

Using this estimate we can now estimate the relative entropies with $F_p$
instead of $\LB$, the technique being exactly the same as above.

\begin{theorem}[Exponential decay of $\calE_p(\bfu(\tau))$] 
\label{thm:Conv.E12}
For $\alpha\geq 1$ choose any $p>0$ with
$ \alpha{-}1 \leq p \leq \max\{\alpha/2, \alpha{-}1\}$ and define the relative
entropy $E_p(\tau) := \calE_{p}(\bfu(\tau) \, \vert \, \bfU)$ for the \AAA
similarity \EEE profile $\bfU=(U,V)^\top$ satisfying
\eqref{eq:Profile.specialCase}.  Then, all solutions $\bfu = (u,v)^\top$ of
the scaled system \eqref{eq:RDS.specialCase} with $E_{p}(0) <\infty$
satisfy the estimate
\begin{align*}
&\frac{\rmd}{\rmd\tau} E_p(\tau) \leq - \Big( \frac12 - \wt
\mu_{p,\alpha}\ee^{-\tau} \Big) E_p(\tau) + \wt K_{p,\alpha} \ee^{-\tau}, 
\end{align*}
where the constants $\wt\mu_{p,\alpha}$ and $\wt K_{p,\alpha}$ are given by
\begin{align*}
  &\text{for }2\leq \alpha=p+1:\quad \AAA \wt\mu_{\alpha-1,\alpha}=0 \text{ and } 
  \wt K_{\alpha-1,\alpha }= \frac{1}{4\,k}  \int_\R
  \frac{\Lambda^2}{U^\alpha} \dd y ,
  \\
  &\text{for }1\leq \alpha <2:\quad \wt\mu_{p,\alpha}=
  \AAA  \frac{\kappa}{k} \,\wh M_{p,\alpha} \EEE \Big\Vert 
  \frac{\Lambda^2 }{ U^{\alpha+1} } \Big\Vert_{\rmL^\infty} \text{ and } 
  \wt K_{p,\alpha}  =  \AAA \frac{1}{ k} \,\wh M_{p,\alpha} \EEE \big(\kappa
  F_p^*(\kappa){+}1\big) \int_\R \frac{\Lambda^2}{U^\alpha} \dd y,
\end{align*}
where $\kappa = \sqrt{1{+}p{-}\alpha}$. 
\end{theorem}
\begin{proof}
We continue with estimate \eqref{eq:Est.Imix.Phi.p.alpha} and obtain 
\begin{align*}
& \Ipmix  (\rho,\zeta) - \ee^\tau  \Dpreact(\rho,\zeta) 
\leq \ee^{-\tau}\, \AAA \wh M_{p,\alpha}  \EEE\int_\R
\rho^{1+p-\alpha}\, \AAA \frac{\Lambda^2}{ k\,U^\alpha}  \EEE \dd y .
\end{align*}
For $\alpha\geq 2$ we have $p=\alpha{-}1$, and the integrand is independent of
$\rho$. Hence, the assertion is clear in this case. 

For $\alpha\in [1,2)$ we observe $\kappa=\sqrt{1{+}p{-}\alpha}\in (0,1]$ and 
estimate as follows:
\[
\AAA \rho^{1{+}p{-}\alpha} = \EEE \rho^{\kappa^2} \leq \kappa^2 \rho +(1{-}\kappa^2) \leq
\kappa\, \big(\kappa\, \rho\big) + 1 \leq \kappa\,\big(
F_p(\rho)+F_p^*(\kappa) \big) +1.
\]
With this we find 
\begin{align*}
& \Ipmix  (\rho,\zeta) - \ee^\tau  \Dpreact(\rho,\zeta) 
\leq \ee^{-\tau}\, \AAA \frac{\wh M_{p,\alpha}}{k} \EEE \bigg( \kappa
\big\|\frac{\Lambda^2}{ U^{\alpha+1}}\big\|_{\rmL^\infty} \calE_p(\bfu) + \big(
  \kappa F_p^*(\kappa){+}1\big) \int_\R\frac{\Lambda^2}{U^\alpha} \dd y\bigg),
\end{align*}
which gives the desired result for $\alpha\in [1,2)$. 
\end{proof}

\AAA The following corollary provides some natural consequences of the above
result. First, we show that for $\alpha \geq 3$ we again have exponential decay
like $\ee^{-\tau/2}$ if we use the relative entropy
$\calE_{\alpha-1}(\bfu|\bfU)$. This improves the result in Theorem
\ref{thm:specialCase}, where $\calE_1=\EB$ only decays like
$\ee^{-\tau/(\alpha-1)}$. Second, we show that for the linear case with
$\alpha=1$ we can significantly improve the constants $\mu_0$ and $K_0$
in Lemma \ref{le:Control.Imix.n1} by using $\calE_{1/2}$ instead of
$\calE_1=\EB$. 

\begin{corollary}[Decay of $\calE_p$ for $\alpha=\beta$]
\label{co:Decay.Ep.al=be} 
We have the following estimates:
\begin{subequations}
\label{eq:Ep.al=be}
\begin{align} 
\label{eq:Ep.al=be.ge2}
\text{For } \alpha\geq 2:\quad & \frac\rmd{\rmd \tau} \,\calE_{\alpha-1}(\bfu|\bfU) \leq -
\frac12 \, \calE_{\alpha-1}(\bfu|\bfU) + \ee^{-\tau} \int_\R
\frac{\Lambda^2}{4k\,U^\alpha} \dd y ,
\\
\label{eq:Ep.al=be=1}
\text{for } \alpha=1:\quad &   \frac\rmd{\rmd \tau} \,\calE_{1/2}(\bfu|\bfU) \leq -\big( 
\frac12- \wt \mu_*\big) \, \calE_{1/2}(\bfu|\bfU) + \ee^{-\tau} \int_\R
\frac{11{+}\sqrt2}{14\,k} \:\frac{\Lambda^2}U \dd y ,
\end{align} 
\end{subequations}
where $\wt\mu_*= \big\|\Lambda/U \big\|^2_{\rmL^\infty}/(k\sqrt8)$. 
\end{corollary} 
\begin{proof}
The estimate in \eqref{eq:Ep.al=be.ge2} is a simple rewriting of the
corresponding case in Theorem \ref{thm:Conv.E12}. 
The second estimate in \eqref{eq:Ep.al=be=1} follows similarly from the case $\alpha=1$
and $p=1/2$ by observing $\kappa=\sqrt{1+p-\alpha}=1/\sqrt{2} $ and $F^*_{1/2}
(\zeta) = 2\zeta/(2{-}\zeta)$. 
\end{proof} 
\EEE

\section{Convergence for the case $\alpha >\beta \geq 1$}
\label{se:GeneralCase}

In this section, we consider the case $\alpha \neq\beta$. Without loss of
generality, we assume that $\alpha >\beta$. We aim to show that solutions to
the Cauchy problem \eqref{eq:transRDS},\eqref{eq:transRDS.BC} converge to the
similarity profile $\bfU =(U,V)^\top$ characterized by the equations
\eqref{eq:SelfSimProfile}, but in contrast to the special case $\alpha=\beta$,
there is no meaningful possibility to simplify the mixed term
$\Imix(\rho,\zeta) = \int_\R \big((\zeta {-} 1 ) \beta {-} ( \rho {-} 1 )
\alpha \big) \Lambda \dd y$.  Let us recall that in this general setting,
$\Lambda$ is given as
\[
\Lambda = \frac{1}{\beta} (d_2 V'' + \frac{y}{2} V')
 = - \frac{1}{\alpha} (d_1 U'' + \frac{y}{2} U') ,
\]
where the last equality follows by \eqref{eq:SelfSimProfile}.  As in Section
\ref{su:al=be.Boltz}, we need to control the mixed term $\Imix$ in order to
estimate the dissipation functional $\calD_\rmB$ from Proposition
\ref{prop:Dissipation}.  The difference to the  case $\alpha=\beta$ is that the mixed
term $\Imix$ cannot be estimated with the reactive dissipation term
$-\ee^\tau \calD_\mathrm{react}$ alone, but we need to steal parts of the bonus
factor $1/2$. More precisely, we do the following
\begin{align}
\Imix (\rho,\zeta) 
&= \int_\R \big((\zeta {-} 1 ) \beta {-} ( \rho {-} 1 ) \alpha \big) \Lambda  \dd y \notag \\
&= \int_\R  
	 \big( \Psi(\zeta^\beta) {-}  \Psi(\rho^\alpha) \big)\Lambda  \dd y 
% \\ & \qquad  
+  \int_\R \Big(  \big(\zeta{-}1 {-} \frac{1}{\beta} \Psi(\zeta^\beta)\big) \beta 
-  \big(\rho{-}1 {-} \frac{1}{\alpha} \Psi(\rho^\alpha) \big)\alpha
    \Big) \Lambda \dd y 
 \notag  \\
&=: \ImOne(\rho, \zeta) + \ImSec(\rho, \zeta) 
 \label{eq:Split.I.mix}
\end{align}
for a suitable function $\Psi$ that will be selected below.
That is, we split the mixed term into the two terms $\ImOne$ and 
$\ImSec$ with the strategy to estimate them separately in the
 following way:
\begin{enumerate}
\item We aim to control the integral term $\ImOne$ with 
$-\ee^\tau \calD_\mathrm{react}$ in a similar way as it is done in Section
 \ref{se:SpecialCase}.
\item The integral term $\ImSec$ will be estimated exploiting 
$-\tfrac12 \EB$. To this end, we need to choose the function $\Psi$ in such a way
 that we can globally estimate
\begin{align}
\label{eq:Psi.glEstimate}
\alpha \,\big\vert \rho {-}1 {-} \frac{1}{\alpha} \Psi(\rho^\alpha)  \big\vert  
	\leq C_{\alpha}  \LB(\rho) \quad 
\text{ and } \quad \beta \, \big\vert \zeta{-}1 {-} \frac{1}{\beta} \Psi(\zeta^\beta) \big \vert 
	\leq C_{\beta}  \LB(\zeta),
\end{align}
for constants $C_{\alpha}, C_{\beta} \geq 0$.  Further, the profile functions
$U$ and $V$ have to satisfy $\Lambda /U \in L^{\infty}(\R)$ and
$\Lambda /V \in L^{\infty}(\R)$.  Even more, we need
\begin{equation*}
\max \Big\{ C_{\alpha} \Vert \Lambda /U \Vert_\infty, 
C_{\beta}  \Vert \Lambda /V \Vert_\infty\Big\} < 1/2,
\end{equation*}
where $1/2$ is the bonus factor. This can be achieved if $A_-$ and $A_+$ are
close enough to each other so that the profiles are flat and hence
$\|\Lambda\|_\infty$ is small enough, we refer to
\cite[Rem.\,5.1]{MieSch21?ESPS}.
\end{enumerate}
The considerations above lead to the function 
\begin{equation*}
\Psi(r) := \max \{\alpha,\beta\} (r^{1/ \max \{\alpha,\beta\}} {-}1) 
  = \alpha(r^{1/\alpha} {-}1).
\end{equation*}
First, it satisfies $\Psi(1) = 0$ and $\Psi'(1) = 1$ for all $\alpha$, which is
a necessary condition for the inequalities \eqref{eq:Psi.glEstimate}.  Second,
it allows us to use the same technique with the Legendre transform in
order to control $\Psi(\zeta^\beta) {-} \Psi(\rho^\alpha)$ through
$\Gamma(\rho^\alpha,\zeta^\beta)$, as we will see later.
Let us start considering point 2. 
If we insert the ansatz for $\Psi$, we obtain
\begin{align*}
%\label{eq:expl.Psi.Est}
\alpha \, \big\vert \rho {-}1 {-} \frac{1}{\alpha} \Psi(\rho^\alpha)  \big\vert  \equiv 0 
\quad \text{ and } 
\beta \, \big\vert \zeta{-}1 {-} \frac{1}{\beta} \Psi(\zeta^\beta) \big \vert 
	=  \Big(\beta{-} \frac{\beta^2}{\alpha}\Big) F_{\beta/{\alpha}}(\zeta)
  	 \leq (\alpha {-}\beta) \LB (\zeta),
\end{align*}
where $F_{\beta/{\alpha}}$ denotes the entropy function defined in \eqref{eq:F.p} for 
$p = \beta/{\alpha}$ so that we can use the property 
$F_p (z) \leq \frac{1}{p} F_1(z) = \frac{1}{p} \LB(z)$, see \eqref{eq:Fp.equiv}.
With this, the following lemma is established.
\begin{lemma} \label{le:Est.Imix2}
For all $\alpha > \beta \geq 1$,  
%the difference $\ImSec - \frac12 \EB$ can be bounded by
%\[
%\ImSec(\rho,\zeta) - \frac12 \EB ( \bfu \, \vert \,  \bfU)  
% \leq - \big(\frac12 {-} \theta\big) \; \EB ( \bfu \, \vert \,  \bfU) ,
%\]
the second addend $\ImSec$ of the mixed term $\Imix$ can be bounded by
\[
\ImSec(\rho,\zeta)  \leq  \theta \; \EB ( \bfu \, \vert \,  \bfU) , \quad \text{ with } \quad
\theta = (\alpha {-} \beta) \,
\Vert \Lambda / V \Vert_{\infty}.
\]
\end{lemma}
After the preliminary considerations, we can formulate the main theorem of this section.
\begin{theorem}[Convergence for $\alpha >\beta \geq 1$]\label{thm:generalCase}
  Consider the relative Boltzmann entropy
  $E(\tau) := \EB(\bfu (\tau) \, \vert \, \bfU)$ for the unique similarity
  profile $\bfU =(U,V)^\top$ which is characterized by \eqref{eq:SelfSimProfile}.
  Assume further that $\vert A_- {-}A_+ \vert$ is small enough such that
\[
\theta:= (\alpha {-} \beta) \,
\Vert \Lambda / V \Vert_{\infty} < 1/2.
\]
Then,  for all solutions $\bfu =(u,v)^\top$ of the Cauchy problem
 \eqref{eq:transRDS},\eqref{eq:transRDS.BC} with finite initial entropy $E(0) < \infty$,
the following differential inequalities hold true:
\begin{subequations}
\label{eq:al.ge.be.Estim}
\begin{align}
\label{eq:al.ge.be1-2.Estim}
&\text{For $\alpha \in (1,2)$, it holds}  
&& \dot E(\tau) 
\leq - \big(\frac12{-} \theta {-} \mu_1 \ee^{-\tau} \big) E(\tau) + K_1 \, \ee^{-\tau} 
\quad \text{ for all } \tau>0,\\
\label{eq:al.ge.be.ge2.Estim}
&\text{and if $\alpha \geq 2$, then }  
&& \dot E(\tau) 
\leq - \big(\frac12{-} \theta \big)  E(\tau) + K_2 \, \ee^{-\tau/(\alpha {-} 1)} \quad \text{ for all } \tau>0,
\end{align}
\end{subequations}
where $\mu_1,K_1$ and $K_2$ were already defined in Section \ref{su:al=be.Boltz}, in
Lemma \ref{le:Control.Imix.ngeq2} and \ref{le:Control.Imix.1n2}, respectively.
\end{theorem}
The proof can be found at the end of this section. Notice that the case $\alpha=1$ 
does not occur since $\alpha > \beta \geq 1$.
Moreover, we see that we obtain here the same constants $\mu_1,K_1$ and $K_2$ 
as in the special case $\alpha = \beta$. But one has to be careful, the constants only
coincide if $\Lambda$ is given in its general form, namely 
$\Lambda = \frac{1}{\beta} (d_2 V'' + \frac{y}{2} V')
 = - \frac{1}{\alpha} (d_1 U'' + \frac{y}{2} U')$, while in Section \ref{se:SpecialCase}
there is the possibility to simplify $\Lambda$ due to $\alpha = \beta$ and $U=V$.

The fact that they coincide, up to the possible simplification of $\Lambda$, 
already indicates that we can trace a part of the proof back to what we already 
did in Section \ref{se:SpecialCase}. 
In the following, we will find out how it works explicitly.

Thus, we turn our attention to point 1 of our strategy.
Luckily, we will see that the function 
$\Psi(r)  = \alpha \left(r^{1/{\alpha}}{-}1\right)$
is a well-working function for estimating $\ImOne$ as well. 
We aim to estimate the difference 
$\ImOne -\ee^\tau \calD_\mathrm{react}$ in a suitable way.
It holds
\begin{align*}
\ImOne(\rho, \zeta)  - \ee^\tau \calD_\mathrm{react}  (\rho,\zeta) 
&=  \int_\R \big(  \Psi(\zeta^\beta){-} \Psi(\rho^\alpha) \big) \Lambda
	-  \ee^\tau k U^\alpha \Gamma(\rho^\alpha, \zeta^\beta) \dd y \\
&=  \int_\R \alpha  \rho
	 \Big(\frac{\zeta^{\beta/{\alpha}}}{ \rho} {-} 1 \Big)  \Lambda
	 - \ee^\tau  k  U^\alpha \rho^\alpha\Big(\frac{\zeta^\beta}{\rho^\alpha} {-}1\Big)
	 \log(\zeta^\beta / \rho^\alpha) \dd y.
\end{align*}
Next, we take the very same function $\Phi_\alpha$ defined in \eqref{eq:Phi.alpha}.
However, in this general case we have to define 
$z:=\frac{\zeta^{\beta/{\alpha}}}{ \rho} -1 $ to ensure everything fits together. 
This yields
\begin{align}
\label{eq:Est.Imix1.phi*}
\ImOne & (\rho,\zeta)   -  \ee^{\tau}  \calD_\mathrm{react}(\rho,\zeta) 
 =  \int_\R \alpha \rho \Lambda   z 
	-  \ee^\tau \, k (U \rho)^\alpha \, \Phi_\alpha(z) \dd y \notag \\
&\leq 	 \int_\R   \ee^\tau \, k (U \rho)^\alpha  \;
	\Phi_\alpha^*\Big(\tilde \Lambda \rho^{1{-}\alpha} \ee^{-\tau} \Big) \dd y,
	\quad \text{ with } \quad 
 \tilde \Lambda:=  \frac{\alpha \Lambda}{ k U^\alpha} 
 = - \frac{d_1 U'' + \frac y2 U'}{k U^\alpha}.
\end{align}
Comparing this estimate above with that
 from Section \ref{se:SpecialCase}, we realize that we end up with the same
 inequality like in \eqref{eq:Estimate.Imix.Phi.n},
with the only difference that $\Lambda$ is in its general setting. With this, we can
 easily proof the following by replicating the steps of Lemma 
 \ref{le:Control.Imix.1n2} and \ref{le:Control.Imix.n1}.
\begin{lemma} \label{le:Est.Imix1}
Under the assumptions of Theorem \ref{thm:generalCase}, for the inequality
 \eqref{eq:Est.Imix1.phi*} the following bounds from below hold true
 for all $\tau >0$:
\begin{subequations}
\begin{align}
&\text{For $\alpha \in (1,2)$, we have}  
&& \ImOne - \ee^{\tau}  \calD_\mathrm{react}
\leq  \mu_1 \ee^{-\tau} \EB(\bfu \, \vert \, \bfU) + K_1 \, \ee^{-\tau} \\
&\text{and for $\alpha \geq 2$, it holds }  
&& \ImOne - \ee^{\tau}  \calD_\mathrm{react}
\leq  K_2 \, \ee^{-\tau/(\alpha {-} 1)},
\end{align}
\end{subequations}
with $\mu_1,K_1$ and $K_2$ precisely defined in
Lemma \ref{le:Control.Imix.1n2} and \ref{le:Control.Imix.n1}, respectively.
\end{lemma}

We are now ready to prove the main theorem of this section.\\
\begin{proof}[Proof of Theorem \ref{thm:generalCase}]
Denote the relative Boltzmann entropy by 
$E(\tau):=  \EB ( \bfu (\tau) \, \vert \,  \bfU)$ and use Proposition 
\ref{prop:Dissipation} together with the nonnegativity of the
Fisher information which yield
\[
 \frac{\d}{\d \tau}E = - \calD_\rmB 
 \quad \text{with }\calD_\rmB \geq \frac12 E + \ee^\tau \, \Dreact   - \Imix.
\]
In this general setting, the mixed term is given by
$\Imix (\rho,\zeta) 
= \int_\R \big((\zeta {-} 1 ) \beta {-} ( \rho {-} 1 ) \alpha \big) \Lambda  \dd y $. 
So, in \eqref{eq:Split.I.mix} we split $\Imix$ into $\ImOne$ and $\ImSec$ and 
estimated both addends separately.
First, using Lemma \ref{le:Est.Imix2} yields
$\ImSec \leq \theta \, E(\tau)$
for all $\alpha > \beta \geq 1$. Second, Lemma \ref{le:Est.Imix1} provides
in both cases $\alpha \in (1,2)$ and $\alpha \geq 2$ that
$\ImOne - \ee^{\tau} \, \Dreact
\leq \mu_j\ee^{-\tau} E(\tau) + K_j \ee^{-\sigma_j \tau}$ for $j= 1,2$, with $\mu_2 =0$.
Inserting all estimates, we obtain \eqref{eq:al.ge.be.Estim} as claimed.
\end{proof}

This proof concludes the section and therefore also completes the new results 
we can formulate on the convergence towards similarity profiles. 
Through this, it provides a qualitative statement 
on the long-time behavior of solutions to the given reaction-diffusion system 
under the difficulty of considering 
the whole space, i.e. unbounded domains, and infinite mass.
In the current paper, new approaches are used which allow to treat these difficulties. 
In particular, this means the transformation of the system into parabolic scaling
variables, which generates the bonus factor $d/2$ and thus allows the 
energy-dissipation estimates (of Boltzmann type), that are well-studied 
on bounded domains, to be applied to the whole space.
In the case $\alpha \neq \beta$, we have seen that additional restrictions
must hold on the similarity profile, namely that $\vert A_- {-} A_+ \vert$ is 
sufficiently small.  The question, whether the convergence holds true for solutions 
with arbitrary boundary values, remains open.
Furthermore, the considered system, especially the reaction 
$\alpha X_1 \rightleftharpoons \beta X_2$, is still very simple. 
Perhaps, the methods presented above provide a starting point for more 
complicated systems admitting a family of steady-states.

\AAA 
Our approach based on energy-dissipation estimates seems to be flexible
enough to treat more complicated reaction-diffusion systems, where 
more species and more reactions can be involved. For the general setup we refer
to  \cite{MieSch21?ESPS} where the existence result for similarity profiles is
proved for these general situations. The main task is then the control of the mixed term
$\Imix$, which may now involve more than one Lagrange multiplier.\EEE

\appendix
\section{Appendices}
\label{se:Appendix}

In this appendix, we will give the omitted proofs of auxiliary results we used
throughout this paper.

\noindent
\begin{proof}[Proof of Lemma \ref{le:Gronwall}]
The given differential inequality is of the form
\[
\frac{\rmd}{\rmd \tau} E(\tau) \leq - a(\tau) \, E(\tau) + b(\tau) \quad 
\text{ for } a(\tau):= \eta - \mu \, \ee^{-\tau} \text{ and } b(\tau):= K \, \ee^{-\gamma \tau}.
\]
A direct application of the differential version of Gr\"onwall's lemma yields
\begin{equation} \label{eq:Est.E}
E(\tau) \leq E(0)\, \ee^{-\int_{0}^\tau a(r) \dd r} 
+ \int_{0}^\tau b(s) \, \ee^{-\int_{s}^\tau a(r) \dd r} \dd s.
\end{equation}
First, we estimate the integral $\int_{s}^\tau a(r) \dd r$ roughly but sufficiently as we 
will see. We have
\begin{align*}
-\int_{s}^\tau a(r) \dd r = \int_{s}^\tau   \mu \, \ee^{-r} {-} \eta \dd r 
=  \mu (\ee^{-s} {-} \ee^{-\tau}) - \eta (\tau {-}s)\leq  \mu -  \eta (\tau {-}s) 
\end{align*}
since $s \geq 0$. Inserting this into \eqref{eq:Est.E} and using the 
monotonicity of the exponential function gives
\begin{align*}
E(\tau) & \leq E(0)\, \ee^{\mu -  \eta \tau} 
+ K \, \ee^{\mu -\eta \tau} \int_{0}^\tau \, \ee^{(\eta {-} \gamma)s } \dd s 
= \ee^{\mu} \Big( E(0)\, \ee^{-  \eta \tau} 
+ \frac{K}{\eta {-} \gamma} \,  \big( \ee^{-\gamma \tau} {-} \ee^{-\eta \tau} \big) \Big) \\
& \leq \ee^{- \min \{\eta,\gamma \}\tau +  \mu} \Big( E(0)
+ \frac{2 K}{\vert \eta {-} \gamma \vert } \Big) \quad \text{ if } \eta \neq \gamma.
\end{align*}
In case $\eta = \gamma$, the integrand above is identically $1$, which gives the result.
\end{proof}

We now give the proof of Lemma \ref{le:phi.n.Legendre}, which provides the
necessary upper estimates of the functions $\Phi_\alpha$. 

\begin{proof}[Proof of Lemma \ref{le:phi.n.Legendre}]
Throughout we use $\alpha\geq 1$ and write 
$\Phi_\alpha =g_\alpha h_\alpha$ with
\[
g_\alpha (z) = \begin{cases} \big|\log(z{+}1)^\alpha\big|& \text{for }z>-1,\\ \infty
  &\text{for }z\leq -1, \end{cases}
\quad \text{ and} \quad 
h_\alpha(z)=  \begin{cases} \big|(z{+}1)^\alpha-1\big|& \text{for }z\geq -1,\\ \infty
  &\text{for } z < -1.  \end{cases}
\] 

To obtain upper bounds for $\Phi_\alpha^*$, we derive lower bounds for
$\Phi_\alpha$. We do the estimates for $z\geq 0$ and
$z\leq 0$ separately.

\underline{For $z\leq 0$} it suffices to consider $z\in (-1,0)$. 
First observe $g_\alpha(z)\geq - \alpha z$ by convexity. Next, we have
$h_\alpha(z)=1-(z{+}1)^\alpha$, which is concave because of $\alpha\geq 1$. Hence,
$h_\alpha  (0)=0$ and $h_\alpha(1)=1$ imply $h_\alpha(z)\geq -z$ for $z\in
[-1,0]$. Together we find 
\[
\Phi_\alpha(z) \geq \alpha z^2\  \text{ for } z\in [-1,0], \qquad
\Phi_\alpha(z)=\infty \ \text{ for }z<-1. 
\]

\underline{For $ z\geq 0$} we use $\log 2 >1/2$ which implies $\log y \geq
\min \big\{ (y{-}1)/2, 1/2 \big\}$ for $y\geq 1$. 
%Setting $x=-z \geq 0$ 
We find 
\[
g_\alpha(z)=\alpha \log(z{+}1) \geq \alpha \min\big\{ z/2, 1/2\big\}.
\]
For $h_\alpha$ we obtain $h_\alpha(z)=(z{+}1)^\alpha - 1 \geq z^\alpha$,
  which follows by observing that the derivative of both sides satisfy the same
  inequality and a subsequent integration. Moreover, $h_\alpha(z) \geq \alpha
  z$ by convexity. Hence, $h_\alpha(z) \geq \max\big\{\alpha z, z^\alpha
  \big\}$, and we arrive at 
\[
\Phi_\alpha(z) \geq \frac\alpha2 \,\max\big\{\alpha |z|, |z|^\alpha \big\} \,
\min\big\{ |z|, 1\big\} \quad \text{for } z\geq 0.
\]

\underline{For $\alpha=1$} the above estimate is too weak. To obtain a better estimate we
observe that $F_0(\rho) =\rho-\log \rho -\rho \geq 0$, see \eqref{eq:F.p}.
Hence, setting $z=\rho-1$ we have 
\[
\Phi_1(z)=(\rho{-}1) \log \rho = \LB(\rho) +F_0(\rho) \geq
\LB(\rho)=\LB(z{+}1) \quad \text{for }z\geq -1.\medskip
\] 

To obtain upper estimates for $\Phi_\alpha^*$ we  
use now again that the Legendre transform is order reversing and the fact that
for a family of functions $(f_i)_{i \in I}$ the equality
$\big( \inf_i (f_i) \big)^* = \sup_i \big( f_i^* \big)$ holds true, which is a direct
consequence of the definition (see \cite[Prop.\,3.50]{Peyp15CONS}). 

For $\alpha=1$ this gives
\[
\Phi_1^*(\xi) \leq \sup\big(\xi z - \LB(z{+}1)\big) = -\xi +
\LB^*(\xi) = \ee^{\xi}-\xi -1 \quad \text{for all } \xi \in \R.
\]
For $\alpha\in [1,2]$ the above estimates give $\Phi_\alpha(z) \geq
\min\big\{ \frac\alpha 2 |z|^\alpha, \frac\alpha 2 z^2\big\}$, and we obtain
\[
\Phi_\alpha^*(\xi) \leq  \max\big\{ \wt c_\alpha |\xi|^{\alpha/(\alpha-1)},
\, \frac1{2\alpha} \,\xi^2\big\}. 
\]
In the case $\alpha\geq 2$ we have $|z|^\alpha\leq |z|^2$ for $|z|\leq 1$ and 
find $\Phi_\alpha(z)\geq \frac\alpha2 |z|^\alpha$, which gives the desired
result. 

Finally, we observe that for all $\alpha\geq 1$ we have $\Phi_\alpha(z)\geq
\frac\alpha2\, z^2$ for $|z|\leq 1$, which implies $\Phi_\alpha^*(\xi) \leq
\frac1{2\alpha} \,\xi^2$ for $|\xi|\leq \alpha$. 
\end{proof}

\paragraph*{Acknowledgments.} The research was partially supported by Deutsche
Forschungsgemeinschaft (DFG) via the Collaborative Research Center SFB\,910
``Control of self-organizing nonlinear systems'' (project number 163436311),
subproject A5 ``Pattern formation in coupled parabolic systems''.

\footnotesize

% \addcontentsline{toc}{section}{References}

\bibliographystyle{alpha_AMs}
\bibliography{alex_pub,bib_alex}

\newcommand{\etalchar}[1]{$^{#1}$}
\def\cprime{$'$}
\providecommand{\bysame}{\leavevmode\hbox to3em{\hrulefill}\thinspace}
\begin{thebibliography}{11}\itemsep0.1em

\bibitem[Ali79]{Alik79AIPR}
N.~D.~Alikakos: \emph{An application of the invariance principle to
  reaction-diffusion equations}. J. Diff. Eqns. \textbf{33}:2 (1979) 201--225.

\bibitem[Ber82]{Bert82ABSN}
M.~Bertsch: \emph{Asymptotic behavior of solutions of a nonlinear diffusion
  equation}. SIAM J. Appl. Math. \textbf{42}:1 (1982) 66--76.

\bibitem[BJ{\etalchar{*}}14]{BLMV14LFBD}
T.~Bodineau, L.~Joel, C.~Mouhot, and C.~Villani: \emph{Lyapunov functionals for
  boundary-driven nonlinear drift-diffusion equations}. Nonlinearity
  \textbf{27}:9 (2014) 2111--2132.

\bibitem[BrK92]{BriKup92RGGL}
J.~Bricmont and A.~Kupiainen: \emph{Renormalization group and the
  {G}inzburg-{L}andau equation}. Comm. Math. Phys. \textbf{150}:1 (1992)
  193--208.

\bibitem[CaT00]{CarTos00ALDS}
J.~A.~Carrillo and G.~Toscani: \emph{Asymptotic {$L^1$}-decay of solutions of
  the prous medium equation to self-similarity}. Indiana Univ. Math. J.
  \textbf{49}:1 (2000) 113--142.

\bibitem[CoE92]{ColEck92SPSG}
P.~Collet and J.-P.~Eckmann: \emph{Solutions without phase-slip for the
  {G}insburg-{L}andau equation}. Comm. Math. Phys. \textbf{145}:2 (1992)
  345--356.

\bibitem[DeF06]{DesFel06EDTE}
L.~Desvillettes and K.~Fellner: \emph{Exponential decay toward equilibrium via
  entropy methods for reaction-diffusion equations}. J. Math. Anal. Appl.
  \textbf{319}:1 (2006) 157--176.

\bibitem[DeF07]{DesFel07EMRD}
\bysame, \emph{Entropy methods for reaction-diffusion systems}, Discrete
  Contin. Dyn. Syst. (suppl). Dynamical Systems and Differential Equations.
  Proceedings of the 6th AIMS International Conference, 2007, pp.~304--312.

\bibitem[FeT17]{FelTan17EECE}
K.~Fellner and B.~Q.~Tang: \emph{Explicit exponential convergence to
  equilibrium for nonlinear reaction-diffusion systems with detailed balance
  condition}. Nonlinear Analysis \textbf{159}:C (2017) 145--180.

\bibitem[GaM98]{GalMie98DMSS}
T.~Gallay and A.~Mielke: \emph{Diffusive mixing of stable states in the
  {G}inzburg-{L}andau equation}. Comm. Math. Phys. \textbf{199}:1 (1998)
  71--97.

\bibitem[GaS22]{GalSli22DREE}
T.~Gallay and S.~Slijep\v{c}evi\'{c}: \emph{Diffusive relaxation to equilibria
  for an extended reaction-diffusion system on the real line}. J. Evol. Eqns.
  \textbf{22}:47 (2022) 1--33.

\bibitem[GGH96]{GlGrHu96FEDR}
A.~Glitzky, K.~Gr{\"o}ger, and R.~H{\"u}nlich: \emph{Free energy and
  dissipation rate for reaction diffusion processes of electrically charged
  species}. Applicable Analysis \textbf{60}:3-4 (1996) 201--217.

\bibitem[Gr{\"o}83]{Grog83ABSC}
K.~Gr{\"o}ger: \emph{Asymptotic behavior of solutions to a class of
  diffusion-reaction equations}. Math. Nachr. \textbf{112} (1983) 19--33.

\bibitem[Gr{\"o}92]{Grog92FEEA}
\bysame, \emph{Free energy estimates and asymptotic behaviour of
  reaction-diffusion processes}, WIAS preprint 20, 1992.

\bibitem[HH{\etalchar{*}}18]{HHMM18DEER}
J.~Haskovec, S.~Hittmeir, P.~A.~Markowich, and A.~Mielke: \emph{Decay to
  equilibrium for energy-reaction-diffusion systems}. SIAM J. Math. Analysis
  \textbf{50}:1 (2018) 1037--1075.

\bibitem[J{\"u}n16]{Jung16EMDP}
A.~J{\"u}ngel, \emph{Entropy methods for diffusive partial differential
  equations}, Springer, Berlin, Heidelberg, 2016.

\bibitem[MaP01]{MarPan01DPSC}
P.~Marcati and R.~Pan: \emph{On the diffusive profiles for the system of
  compressible adiabatic flow through porous media}. SIAM J. Math. Analysis
  \textbf{33}:4 (2001) 790--826.

\bibitem[MHM15]{MiHaMa15UDER}
A.~Mielke, J.~Haskovec, and P.~A.~Markowich: \emph{On uniform decay of the
  entropy for reaction-diffusion systems}. J. Dynam. Diff. Eqns.
  \textbf{27}:3-4 (2015) 897--928.

\bibitem[Mie11]{Miel11GSRD}
A.~Mielke: \emph{A gradient structure for reaction-diffusion systems and for
  energy-drift-diffusion systems}. Nonlinearity \textbf{24} (2011) 1329--1346.

\bibitem[Mie17]{Miel17UEDR}
\bysame, \emph{Uniform exponential decay for reaction-diffusion systems with
  complex-balanced mass-action kinetics}, Pattern of Dynamics (P.~Gurevich,
  J.~Hell, B.~Sandstede, and A.~Scheel, eds.), Springer Proc.\ in Math.\ \&
  Stat.\ Vol.\,205, Springer, 2017, pp.~149--171.

\bibitem[MiM18]{MieMit18CEER}
A.~Mielke and M.~Mittnenzweig: \emph{Convergence to equilibrium in
  energy-reaction-diffusion systems using vector-valued functional
  inequalities}. J. Nonlinear Sci. \textbf{28}:2 (2018) 765--806.

\bibitem[MiS23a]{MieSch21?ESPS}
A.~Mielke and S.~Schindler: \emph{Existence of similarity profiles for systems
  of diffusion equations}. Preprint arXiv2301.10360 (2023) .

\bibitem[MiS23b]{MieSch23?SSPC}
\bysame: \emph{Self-similar pattern in coupled parabolic systems as
  non-equilibrium steady states}. Chaos (2023) , Submitted WIAS Preprint 2992,
  arXiv2302.00393.

\bibitem[Mit18]{Mitt18EMQC}
M.~R.~Mittnenzweig, \emph{Entropy methods for quantum and classical evolution
  equations}, Ph.D. thesis, Humboldt-Universit\"at zu Berlin,
  Mathematisch-Naturwissenschaftliche Fakult\"at, 2018.

\bibitem[MP{\etalchar{*}}17]{MPPR17NETP}
A.~Mielke, R.~I.~A.~Patterson, M.~A.~Peletier, and D.~R.~M.~Renger:
  \emph{Non-equilibrium thermodynamical principles for chemical reactions with
  mass-action kinetics}. SIAM J. Appl. Math. \textbf{77}:4 (2017) 1562--1585.

\bibitem[MPR14]{MiPeRe14RGFL}
A.~Mielke, M.~A.~Peletier, and D.~R.~M.~Renger: \emph{On the relation between
  gradient flows and the large-deviation principle, with applications to
  {M}arkov chains and diffusion}. Potential Analysis \textbf{41}:4 (2014)
  1293--1327.

\bibitem[MSU01]{MiScUe01SDDE}
A.~Mielke, G.~Schneider, and H.~Uecker, \emph{Stability and diffusive dynamics
  on extended domains}, Ergodic Theory, Analysis, and Efficient Simulation of
  Dynamical Systems (B.~Fiedler, ed.), Springer--Verlag, 2001, pp.~563--583.

\bibitem[Pel71]{Pele71ABSP}
L.~A.~Peletier: \emph{Asymptotic behavior of solutions of the porous media
  equation}. SIAM J. Appl. Math. \textbf{21} (1971) 542--551.

\bibitem[Pey15]{Peyp15CONS}
J.~Peypouquet, \emph{Convex optimization in normed spaces: Theory, methods and
  examples}, Springer Briefs in Optimization, Springer, 2015.

\bibitem[PSZ17]{PiSuZo17ABRD}
M.~Pierre, T.~Suzuki, and R.~Zou: \emph{Asymptotic behavior of solutions to
  chemical reaction–diffusion systems}. J. Math. Anal. Appl. \textbf{450}
  (2017) 152--168.

\bibitem[Smo94]{Smol94SWRD}
J.~Smoller, \emph{{Shock Waves and Reaction-Diffusion Equations}}, Springer,
  1994.

\bibitem[{va}P77]{VanPel77ABSN}
C.~J.~{van Duyn} and L.~A.~Peletier: \emph{Asymptotic behaviour of solutions of
  a nonlinear diffusion equation}. Arch. Rational Mech. Anal. \textbf{65}
  (1977) 363--377.

\bibitem[V{\'a}z07]{Vazq07PMEM}
J.~L.~V{\'a}zquez, \emph{The porous medium equation. mathematical theory},
  Oxford: Clarendon Press, 2007.

\bibitem[Vol14]{Volp14EPDE}
V.~Volpert, \emph{Elliptic partial differential equations. volume 2:
  Reaction-diffusion equations}, Monographs in Mathematics, Springer Basel,
  2014.

\bibitem[VVV94]{VoVoVo94TWSP}
A.~Volpert, V.~Volpert, and V.~Volpert, \emph{Traveling wave solutions of
  parabolic systems}, American Mathematical Society, 1994.

\end{thebibliography}

\end{document}